\documentclass[a4paper, 11pt]{article}

\usepackage[utf8]{inputenc}
\usepackage{amsmath,amsthm,amssymb}
\usepackage{amsfonts}
\usepackage{mathrsfs}
\usepackage{graphicx}
\usepackage{url}
 \usepackage{multicol}
\usepackage{tabu}
\usepackage[table]{xcolor}
\usepackage{tikz}
\usetikzlibrary{shapes.geometric}
\usetikzlibrary{decorations.markings}
\usetikzlibrary{decorations.pathreplacing}
\usetikzlibrary{arrows,shapes,positioning,backgrounds}
\usetikzlibrary{arrows.meta,decorations.markings}
\usetikzlibrary{calc,trees}
\usepackage{nameref} 
\usepackage{caption}
\usepackage[labelformat=simple,labelfont={}]{subcaption}

\usepackage[top=3.2cm, bottom=3cm, left=3.45cm, right=3.45cm]{geometry}
\usepackage{mathtools}
\usepackage[shortlabels]{enumitem}
\usepackage{color}
\usepackage[breaklinks]{hyperref}
\usepackage{stmaryrd}
\usepackage{calc}
\usepackage{multirow}
\usepackage{arydshln}
\usepackage[noblocks]{authblk}
\usepackage{skak,scalerel}

\DeclareMathOperator{\Bal}{Bal}
\DeclareMathOperator{\Cay}{Cay}
\DeclareMathOperator{\Prim}{Prim}

\DeclareMathOperator{\id}{id}
\DeclareMathOperator{\Par}{Par}
\DeclareMathOperator{\Der}{Der}

\DeclareMathOperator{\Img}{Im}

\newcommand{\LL}{\mathbb{L}}
\newcommand{\BB}{\mathbb{B}}
\DeclareMathOperator{\Alt}{Alt}
\DeclareMathOperator{\AltC}{\overline{\Alt}}
\newcommand{\Cyc}{\mathcal{C}}
\newcommand{\Sym}{\mathcal{S}}
\newcommand{\Even}[1]{#1_{\mathrm{even}}}
\newcommand{\Odd}[1]{#1_{\mathrm{odd}}}
\newcommand{\std}{\mathrm{std}}
\newcommand{\ltrmin}{\textsc{lmin}}
\newcommand{\wltrmin}{\textsc{wlmin}}
\DeclareMathOperator{\filling}{\textsc{fill}}
\newcommand{\wking}{\scalerel*{\WhiteKingOnWhite}{Xg}}
\newcommand{\bking}{\scalerel*{\BlackKingOnWhite}{Xg}}
\newcommand{\wrook}{\scalerel*{\WhiteRookOnWhite}{Xg}}
\newcommand{\brook}{\scalerel*{\BlackRookOnWhite}{Xg}}
\newcommand{\wknight}{\scalerel*{\WhiteKnightOnWhite}{Xg}}
\newcommand{\bknight}{\scalerel*{\BlackKnightOnWhite}{Xg}}
\newcommand{\wpawn}{\scalerel*{\WhitePawnOnWhite}{Xg}}
\newcommand{\bpawn}{\scalerel*{\BlackPawnOnWhite}{Xg}}

\definecolor{orange}{RGB}{255,102,0}
\definecolor{ggreen}{RGB}{0,153,0}
\definecolor{darkblue}{RGB}{0,0,255}
\definecolor{purple}{RGB}{153,51,255}
\definecolor{turq}{RGB}{72,209,204}
\definecolor{gray}{RGB}{220,220,220}
\definecolor{orange2}{RGB}{255,100,0}
\definecolor{purple2}{RGB}{159,51,250}
\definecolor{rred}{rgb}{0.9, 0.17, 0.31}
\definecolor{naugreen}{cmyk}{.43,0,.34,.38}
\definecolor{naublue}{cmyk}{1,.72,0,.32}
\definecolor{mediterranean}{cmyk}{.67,0,.08,.3}
\definecolor{rose}{cmyk}{0,1.00,.20,0}
\definecolor{darkorchid}{cmyk}{.6,.9,0,.05}
\definecolor{butterfly}{cmyk}{.95,.59,0,.10}
\definecolor{springgreen}{cmyk}{1.00,0,.70,.02}
\definecolor{darkred}{cmyk}{0,1,1,.5}
\definecolor{nectarine}{cmyk}{0,0.70,1.00,0}
\definecolor{icyblue}{cmyk}{.84,.25,0,.06}
\definecolor{manatee}{rgb}{0.59, 0.6, 0.67}

\theoremstyle{definition}
\newtheorem{theorem}{Theorem}[section]
\newtheorem{corollary}[theorem]{Corollary}
\newtheorem{lemma}[theorem]{Lemma}
\newtheorem{conjecture}[theorem]{Conjecture}

\newtheorem{proposition}[theorem]{Proposition}
\newtheorem{example}[theorem]{Example}

\begin{document}

\title{Pattern-avoiding Cayley permutations\\ via combinatorial species}

\author[1]{Giulio Cerbai\thanks{G.C. is member of the Gruppo
Nazionale Calcolo Scientifico--Istituto Nazionale di Alta
Matematica (GNCS-INdAM).}}
\author[1]{Anders Claesson}
\author[2]{Dana C.~Ernst}
\author[2]{Hannah Golab}
\affil[1]{Department of Mathematics, University of Iceland,
Reykjavik, Iceland,
\texttt{akc@hi.is},
\texttt{giulio@hi.is}.
}
\affil[2]{Department of Mathematics and Statistics,
Northern Arizona University, 
Flagstaff, AZ,
\texttt{Dana.Ernst@nau.edu},
\texttt{Hannah.Golab@nau.edu}.
}
\date{}

\maketitle
\begin{abstract}
A Cayley permutation is a word of positive integers such that if a letter
appears in this word, then all positive integers smaller than that letter also appear.
We initiate a systematic study of pattern avoidance on Cayley permutations
adopting a combinatorial species approach. Our methods lead to species
equations, generating series, and counting formulas for Cayley permutations
avoiding any pattern of length at most three.
We also introduce the species of primitive structures as a generalization of Cayley permutations with no ``flat steps''. Finally, we explore 
various notions of Wilf equivalence arising in this context.
\end{abstract}

\section{Introduction}
\thispagestyle{empty} 

A permutation $w$ (written in one-line notation) is said to contain a permutation 
$p$ as a pattern if some subsequence of the entries of $w$ has the same relative 
order as all of the entries of $p$. If $w$ does not contain $p$, then $w$ is said 
to avoid $p$. One of the first notable results in the field of permutation patterns 
was obtained by MacMahon~\cite{MacMahon1915} in 1915 when he proved that the Catalan 
numbers enumerate the 123-avoiding permutations. The study of permutation patterns 
began receiving focused attention following Knuth's introduction of stack-sorting in 
1969~\cite{Knuth1969}. Knuth proved that a permutation can be sorted by a stack if 
and only if it avoids the pattern 231. The Catalan numbers also enumerate the 
stack-sortable permutations. The first explicit systematic treatment of pattern 
avoidance was conducted by Simion and Schmidt~\cite{simionschmidt}.
In subsequent years, the notion of pattern avoidance has been extended to 
numerous combinatorial objects, including  
set partitions~\cite{Duncan2011, Jelinek2013, Sagan2006},
multiset permutations~\cite{Heubach2006, Savage2006}, 
compositions~\cite{Heubach2006, k-ary, Savage2006},
ascent sequences~\cite{Conway2022,Duncan2011,Egge2022},
and modified ascent sequences~\cite{modasc,modasc2}.
We refer the reader to the books by Bóna~\cite{Bona2022} and Kitaev~\cite{Kitaev2011}
for a comprehensive summary of pattern avoidance in permutations and words.

In this paper, we study pattern avoidance in the setting of Cayley
permutations, which were so named by Mor and Fraenkel in
1983~\cite{MorFraenkel1984}. A Cayley permutation $w$ is a word of
positive integers such that if a letter $b$ appears in $w$ then all
positive integers $a<b$ also appear in $w$. Let us provide some
background on why these words may be interesting and how they got their
name. In a short article from 1895, Cayley~\cite{Cayley1859} counted a
loosely-defined class of trees, and the following is an interpretation
of their definition. The class in question consists of unlabeled rooted
plane trees with a fixed number of leaves. In addition, all leaves are
equidistant from the root, and the number of nodes at distance $i+1$
from the root is either zero or larger than the number of nodes at
distance $i$ from the root. For instance, these are all such trees with three
leaves:
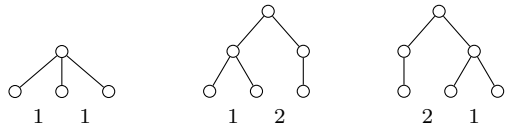
\[
\begin{tikzpicture}[
  scale=0.8,
  level distance=4ex,
  level 1/.style={sibling distance=2em},
  ]
  \begin{scope}[every node/.style = {draw, circle, inner sep=1.6pt}]
  \node {}
    child {node (a) {}}
    child {node (b) {}}
    child {node (c) {}};
  \end{scope}
  \path (a) -- (b) node[midway, yshift=-2ex] {$\scriptstyle 1$};
  \path (b) -- (c) node[midway, yshift=-2ex] {$\scriptstyle 1$};
\end{tikzpicture}\qquad\quad
\begin{tikzpicture}[
  scale=0.8,
  level distance=4ex,
  level 1/.style={sibling distance=3em},
  level 2/.style={sibling distance=2em}
  ]
  \begin{scope}[every node/.style = {draw, circle, inner sep=1.6pt}]
  \node {}
    child {node {}
      child {node (a) {}}
      child {node (b) {}}
    }
    child {node {}
      child {node (c) {}}
    };
  \end{scope}
  \path (a) -- (b) node[midway, yshift=-2ex] {$\scriptstyle 1$};
  \path (b) -- (c) node[midway, yshift=-2ex] {$\scriptstyle 2$};
\end{tikzpicture}\qquad\quad
\begin{tikzpicture}[
  scale=0.8,
  level distance=4ex,
  level 1/.style={sibling distance=3em},
  level 2/.style={sibling distance=2em}
  ]
  \begin{scope}[every node/.style = {draw, circle, inner sep=1.6pt}]
  \node {}
    child {node {}
      child {node (a) {}}
    }
    child {node {}
      child {node (b) {}}
      child {node (c) {}}
    };
  \end{scope}
  \path (a) -- (b) node[midway, yshift=-2ex] {$\scriptstyle 2$};
  \path (b) -- (c) node[midway, yshift=-2ex] {$\scriptstyle 1$};
\end{tikzpicture}\vspace{-1.5ex}
\]
Written between each pair of adjacent leaves is the shortest
distance to a common ancestor of the two leaves. The resulting words are
exactly the Cayley permutations of length two: $11$, $12$, and
$21$. This connection between the trees and the words seems to have been first noted
by MacMahon~\cite{MacMahon1890}, though he formulated it in terms of
compositions of a ``multipartite number''.
Cayley derived the exponential generating series
\[
  \frac{1}{2-e^x} \,=\,
    1 + x + 3\cdot\frac{x^2}{2!} + 13\cdot\frac{x^3}{3!}
    + 75\cdot\frac{x^4}{4!} + 541\cdot\frac{x^5}{5!} + \cdots
\]
for the number of trees with $n+1$ leaves, which is the same as the number
of Cayley permutations of length $n$. The coefficients $1$, $1$,
$3$, $13$, $75$, $541$, $\dots$ of this series are known as the Fubini
numbers, and they appear as entry \href{https://oeis.org/A000670}{A000670}
in the Online Encyclopedia of Integer Sequences (OEIS)~\cite{oeis}.

There are quite a few authors building on the work of Mor and
Fraenkel~\cite{MorFraenkel1984}. They include
Fraenkel and Mor~\cite{FM1983},
Fraenkel~\cite{Fraenkel1985},
Bernstein and Sloane~\cite{BS1995},
Göbel~\cite{Gobel1997},
Hoffman~\cite{Hoffman2001, Hoffman2016},
Baril~\cite{Baril2003},
Jacques, Wong, and Woo~\cite{Jacques-etal},
M\"utze~\cite{Mutze2023},
Cerbai~\cite{sortcay,modasc,modasc2},
Cerbai and Claesson~\cite{CC:fish, CC:tpb, CC:caypol},
and Cerbai, Claesson and Sagan~\cite{modasc3}.
This list is not complete; it is not even complete when restricting to
authors using the Cayley permutation terminology. It is even less
complete when the multitude of different guises that Cayley permutations have
appeared under are taken into account. Synonyms for Cayley permutation include
\emph{packed word}~\cite{FNJ2014, KR2022, Marberg2020, Marberg2021},
\emph{surjective word}~\cite{Hazewinkel2008},
\emph{Fubini word}~\cite{BilleyRyan, Rhoades2022, Wilson2018}, and
\emph{initial word}~\cite{PP2016}.

Cayley permutations may be seen as representatives for equivalence
classes of words modulo order isomorphism. For instance, if we take the
ambient space of words to be $\{1,2,3\}^3$, then the equivalence class
represented by $111$ is $\{111,222,333\}$, while $112$ represents $\{112, 223\}$.
This hints at the view that Cayley permutations simply are \emph{patterns}~\cite{Savage2006}
or (in statistics) \emph{generalized ordinal patterns}~\cite{SF2022}.

A \emph{weak order}~\cite{Bailey1998, Birmajer-etal, HH2013, Jacques-etal},
also called a
\emph{preferential arrangement}~\cite{AUP2013, BZG2016, Gross1962, NBCC2020, Pippenger2010}
or a \emph{race with ties}~\cite{Mendelson1982},
is a way to order
objects where ties are allowed: it is a reflexive, transitive, and
total binary relation. To each Cayley permutation $w=w_1w_2\ldots w_n$
there is an associated binary relation $\preceq$ defined by
$i\preceq j\Leftrightarrow w_i\leq w_j$. This sets up a one-to-one
correspondence between Cayley permutations and weak orders.

Cayley permutations encode \emph{ballots}, also known as \emph{ordered set partitions}:
if $\varpi$ is a ballot on $\{1,\ldots,n\}$, then the
corresponding Cayley permutation $w$ has $n$ letters and its $i$th
letter equals the unique index $j$ such that $i$ belongs to the $j$th
block of $\varpi$.
While pattern avoidance has been studied in the context of
ballots~\cite{Birmajer-etal, chen, Godbole2014, KASRAOUI}, the notion of pattern avoidance
we explore in this paper is distinct.
The interplay between (equivalence classes of) pattern-avoiding Cayley permutations
and modified ascent sequences was explored by Cerbai and Claesson~\cite{CC:tpb}.
The same authors introduced Caylerian polynomials~\cite{CC:caypol}, a generalization
of the Eulerian polynomials that tracks the number
of descents over Cayley permutations.

Most of the structures listed so far can be encoded as Cayley 
permutations satisfying additional properties. For instance, modified ascent sequences
are Cayley permutations where
an entry is an ascent top if and only if it is the leftmost copy of the
corresponding integer in the sequence~\cite{CC:tpb}.
Another classical example is given by Stirling permutations~\cite{GS78},
defined as $212$-avoiding Cayley permutations in which each letter appears twice. 

Our approach to pattern-avoiding Cayley permutations
involves the use of combinatorial species, defined in
Section~\ref{sec:Species}, as formal objects that capture both their structural
and enumerative properties.
In Section~\ref{sec:cayleyperms}, we define the species of 
Cayley permutations and in Section~\ref{sec:pattsonCayley}, we introduce the 
notion of pattern avoidance for Cayley permutations. These sections lay the 
groundwork for the remainder of the paper. In Sections~\ref{sec: length two} 
and~\ref{sec: length three}, we provide species
descriptions that lead to generating series and counting formulas for Cayley 
permutations
avoiding any pattern of length two and three. Table~\ref{tab: patterns} summarizes
our main results. 
In Section~\ref{sec: primitive}, we introduce the species of primitive structures as a generalization of Cayley permutations with no ``flat steps'' and study pattern avoidance in this context. 
The paper concludes in Section~\ref{section: wilfeq} with an exploration of various 
notions of Wilf equivalence in connection with Cayley permutations. 

The work contained in this paper was initiated in the fourth author's 2024 master's 
thesis~\cite{Golab2024}. In an upcoming companion paper, we will tackle sets of  patterns.

\section{Species}\label{sec:Species}

In this section, we mimic the development of species in
the book by Bergeron, Labelle and Leroux~\cite{holybook}. See also
Claesson's short introduction to the topic~\cite{12foldway}.
We utilize two different types of species, namely
$\BB$-species and $\LL$-species.  Loosely speaking, a species
defines both a class of (labeled) combinatorial objects and how those
objects are impacted by relabeling. This mechanism of relabeling is
called the transport of structure.

A \emph{$\BB$-species} (or simply \emph{species}) $F$ is a rule
that produces
\begin{itemize}
\item for each finite set $U$, a finite set $F[U]$;
\item for each bijection $\sigma: U \rightarrow V$, a bijection
  $F[\sigma]: F[U] \rightarrow F[V]$ such that
  $F[\sigma \circ \tau] = F[\sigma]\circ F[\tau]$ for all bijections
  $\sigma: U \rightarrow V$, $\tau: V \rightarrow W$, and
  $F[\id_U] = \id_{F[U]}$ for the identity map $\id_U: U \rightarrow U$.
\end{itemize}

An element $s \in F[U]$ is called an \emph{$F$-structure} on $U$ and the
function $F[\sigma]$ is called the \emph{transport of $F$-structures
  along $\sigma$}, or simply \emph{transport of structure} if the
context is clear. 

In the language of category theory, a $\BB$-species is a functor
$F: \BB \to \BB$, where $\BB$ is the category of
finite sets with bijective functions as morphisms.

Below we list several species that will be used throughout this
paper.
\begin{multicols}{2}
\begin{enumerate}[(a)]
\item $E$: sets;
\item\label{E_{even} species} $\Even{E}$: sets of even cardinality;
\item $\Odd{E}$: sets of odd cardinality;
\item $X$: singletons;
\item $1$: characteristic of empty set;
\item $L$:  linear orders;
\item $\Sym$: permutations;
\item $\Cyc$: cyclic permutations;
\item $\Par$: set partitions;
\item $\Bal$: ballots.
\end{enumerate}
\end{multicols}

Let $F$ and $G$ be species. Then $G$ is said to be a \emph{subspecies}
of $F$, and we write $G\subseteq F$, if it satisfies the following two
conditions:
\begin{itemize}
\item for each finite set $U$, we have $G[U]\subseteq F[U]$;
\item for any bijection $\sigma:U\to V$, we have $G[\sigma]=F[\sigma]|_{G[U]}$.
\end{itemize}

For a species $F$, we let $F_+$, $\Even{F}$, $\Odd{F}$, and $F_n$ denote 
the subspecies of $F$ consisting of $F$-structures on nonempty sets, 
sets with even cardinality, sets with odd cardinality, and sets 
with cardinality $n$, respectively.

Throughout this paper, we define $[n] = \{1,2,\ldots,n\}$ for $n\geq 0$, where $[0]=\emptyset$.

\begin{example}[Linear orders]\label{example of linear orders}
  Let $U$ be a finite set with $n = |U|$ elements. We identify a linear
  order of $U$ with a bijection $f: [n] \to U$, which we may represent
  using one-line notation: $f(1)\ldots f(n)$. Letting $B^A$ denote the set
  of functions from $A$ to $B$, we have
  $L[U] = \{ f\in U^{[n]} : f \text{ bijection}\}$. Note that
  $|L[U]|=n!$. The transport of structure along a bijection
  $\sigma:U \to V$ is given by $L[\sigma](f) = \sigma \circ f$ or, in
  one-line notation, $L[\sigma](f)= \sigma(f(1)) \ldots \sigma(f(n))$.
\end{example}

\begin{example}[Permutations]\label{example of permutations}
  Next we describe the species $\Sym$ of permutations.
  The $\Sym$-structures are defined by
  $\Sym[U] = \{f\in U^U: f \text{ bijection}\}$ while the
  corresponding transport of structure along a bijection $\sigma:U \to V$
  is given by $\Sym[\sigma](f) = \sigma \circ f \circ \sigma^{-1}$.
  This reflects the fact that
  conjugation preserves the cycle type of a permutation. Note that
  $|\Sym[U]|=n!$ when $|U|=n$.
\end{example}

To aid in the enumeration of $F$-structures, we associate an exponential
generating series, denoted by $F(x)$. It is easy to see that for any finite set $U$, the
number of $F$-structures on $U$ depends only on the number of elements
of $U$. For ease of notation, we use $F[n] = F[[n]]$.
We define the \emph{(exponential) generating series} of the species $F$ to be the formal power series
\[
   F(x) = \sum_{n \geq 0} |F[n]| \frac{x^n}{n!}.
\]
The generating series associated with some of the species
introduced above are provided below.
\begin{multicols}{2}
\begin{enumerate}[(a)]
  \item\label{E egf} $E(x)= e^x$;
  \item $\Even{E}(x)= \cosh(x)$;
  \item $\Odd{E} = \sinh(x)$;
  \item $X(x) = x$;
  \item $1(x) = 1$;
  \item\label{L egf} $L(x)=1/(1-x)$;
  \item $\Sym(x) = 1/(1-x)$;
  \item $\Cyc(x) = -\log(1-x)$;
  \item $\Par(x) = \exp(e^x -1)$;
  \item $\Bal(x)=1/(2-e^x)$.
  \end{enumerate}
\end{multicols}

Two species $F$ and $G$ are \emph{isomorphic} if there
is a family of bijections
\[
  \alpha_U: F[U] \to G[U]
\]
such that for
any bijection $\sigma: U \to V$ between two finite sets, the diagram in
Figure~\ref{fig: isomorphic as species diagram} commutes. That is,
$\alpha_V \circ F[\sigma] = G[\sigma]\circ \alpha_U$. In the language of
category theory, $F$ and $G$ are isomorphic if and only if there exists a natural
isomorphism between the functors $F$ and $G$. As is tradition, we will
consider two species as equal if they are isomorphic and use the notation
$F=G$ for both concepts.
Clearly, if $F=G$, then $F(x)=G(x)$. It is important to note that, for $\BB$-species,
the converse of this implication is
false. Recall the species $L$ (linear orders) and $\Sym$ (permutations) from
Examples~\ref{example of linear orders} and~\ref{example of permutations}.
Despite $L$ and $\Sym$ not being isomorphic, it should come as no
surprise that they have the same generating series.

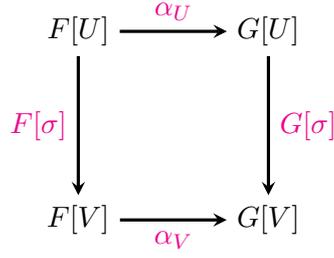
\begin{figure}
  \centering
  \begin{tikzpicture}
    [every circle node/.style={draw, circle, inner sep=1.25pt},scale=1.25]
    \node (a) at (0,0) {$F[V]$};
    \node (b) at (0,2) {$F[U]$};
    \node (c) at (2,0) {$G[V]$};
    \node (d) at (2,2) {$G[U]$};

    \draw [black,-stealth, very thick] (b) to node[midway, left, magenta]
{$F[\sigma]$} (a);
    \draw [black,-stealth, very thick] (b) to node[midway, above, magenta]
{$\alpha_U$} (d);
    \draw [black,-stealth, very thick] (d) to node[midway, right, magenta]
{$G[\sigma]$} (c);
    \draw [black,-stealth, very thick] (a) to node[midway, below, magenta]
{$\alpha_V$} (c);
  \end{tikzpicture}
  \caption{Commutative diagram for isomorphic species.}
  \label{fig: isomorphic as species diagram} 
\end{figure}

One can build new species with operations on previously known
species. For species
$F$ and $G$, an \emph{$(F+G)$-structure} is either an $F$-structure or
a $G$-structure. Denoting disjoint union by $\sqcup$, we have
\[
  (F+G)[U]=F[U] \sqcup G[U]
\]
and for all bijections $\sigma: U \to V$,
\begin{equation*}
  (F+G)[\sigma](s) = 
  \begin{cases}
    F[\sigma](s) &  \text{if }s \in F[U], \\
    G[\sigma](s) &  \text{if }s \in G[U].
  \end{cases}
\end{equation*}
It is easy to see that $(F+G)(x)=F(x)+G(x)$.
 
For species $F$ and $G$, we define an \emph{$(F \cdot G)$-structure} on a
finite set $U$ to be a pair $(s,t)$ such that $s$ is an $F$-structure on a
subset $U_1 \in U$ and $t$ is a $G$-structure on
$U_2 = U \setminus U_1$. Formally,
\[(F \cdot G)[U] = \bigsqcup_{(U_1,U_2)} F[U_1] \times G[U_2]\] with
$U=U_1 \sqcup U_2$.
Transport of structure is defined by
\[(F \cdot G)[\sigma](s,t) = (F[\sigma_1](s), G[\sigma_2](t)),\] where
$\sigma_1 = \sigma |_{U_1}$ and $\sigma_2 = \sigma |_{U_2}$. It turns out
that $(F \cdot G)(x) = F(x) \cdot G(x)$.

Now we define the composition of species. For species $F$ and $G$, 
an
\emph{$(F \circ G)$-structure} is a generalized partition in which each
block of a partition carries a $G$-structure and blocks are structured
by $F$. Formally, if $F$ and $G$ are two species such that
$G[\emptyset] = \emptyset$, we define
\[\displaystyle (F \circ G)[U] =
  \bigsqcup_{\beta=\{B_1,\ldots, B_k\}}
  F[\beta] \times G[B_1] \times \cdots \times G[B_k],
\]
where $\beta=\{B_1,\ldots, B_k\}$ is a partition of $U$. The
details on the transport of structure can be found in Chapter~1
of the book by Bergeron, Labelle and Leroux~\cite{holybook} or on
page 521 of Claesson~\cite{12foldway}.
As expected, it is also true that $(F \circ G)(x) = F(G(x))$.

\begin{example}[Cyclic permutations]\label{example of cyclic permutations}
  Any permutation can be written as a product of disjoint cycles. Since
  the cycles are disjoint they commute and can be written in any
  order. In other words, a permutation can be viewed as a set of 
  cycles. In terms of composition of species we have
  \[
    \Sym = E(\Cyc),
  \]
  where $\Cyc$ is the species of \emph{cyclic permutations}; that is,
  permutations that when written as a product of disjoint cycles
  consist of a single cycle.
\end{example}

\begin{example}[Ballots]\label{example of ballots}
A \emph{ballot}
on a finite set $U$ is a sequence of sets $B_1B_2\ldots B_k$, where each 
$B_i$ is a nonempty subset of $U$, $B_i \cap B_j \neq \emptyset$ for 
$i \neq j$, and $\bigcup_{i=1}^k B_i = U$. Each $B_i$ is referred to as 
a \emph{block}. The species of ballots is given by $\Bal = L(E_+)$.
For a bijection $\sigma: U \to V$, the transport of structure
$\Bal[\sigma]: \Bal[U] \to \Bal[V]$ is given by
\[
\Bal[\sigma] (B_1B_2\ldots B_k ) = \sigma(B_1)\sigma(B_2)\ldots \sigma(B_k),
\]
where $\sigma(B_i)=\{\sigma(x) : x \in B_i\}$. It follows that
\[
\Bal(x) = L(E_+(x)) = \frac{1}{1-(e^x-1)} = \frac{1}{2-e^x}.
\]
\end{example}

The examples of species we have met thus far are well known~\cite{holybook}.
In contrast, the species specification of the alternating group given in
the next example is, as far as we know, new.

\begin{example}[Alternating group]\label{example: the alternating group}
  Let a permutation $f\in \Sym[n]$ be given and assume that when written
  as a product of disjoint cycles it has $c$ cycles. It is easy to see
  that the parity of $f$ is the same as the parity of $n-c$. Hence
  $f$ is even if and only if $n$ and $c$ are both even or $n$ and $c$
  are both odd. Thus,
  \[
    \Alt =
    \Even{(\Even{E}\circ \Cyc)} +
    \Odd{(\Odd{E}\circ\Cyc)}
  \]
  is a species specification of the \emph{alternating group}.
  The integers are commonly defined as equivalence classes of pairs of
  natural numbers, where $(a,b)\sim (c,d)\Leftrightarrow
  a+d=b+c$. Virtual species~\cite[\S 2.5]{holybook} is a similar extension
  allowing for the subtraction of species.
  Using virtual species, for any species $F$, we have
  \begin{align*}
  \Even{F} & = F_0+F_2+F_4+\cdots = \frac{1}{2}\bigl(F + F(-X)\bigr);\\
  \Odd{F} & = F_1+F_3+F_5+\cdots = \frac{1}{2}\bigl(F - F(-X)\bigr).
  \end{align*}
  For ease of notation, let
  $\Sym_e=\Even{E}(\Cyc)$
  and
  $\Sym_o=\Odd{E}(\Cyc)$
  be the species of permutations with an even and odd number of cycles,
  respectively.  Then,
  \begin{align}
    \Alt
    &= \Even{(\Sym_e)} + \Odd{(\Sym_o)} \nonumber \\[1ex]
    &= \frac{1}{2}\bigl(\Sym_e + \Sym_e(-X)\bigr)
      + \frac{1}{2}\bigl(\Sym_o - \Sym_o(-X)\bigr) \nonumber \\[1ex]
    &= \frac{1}{2}\bigl(S + \Sym_e(-X)-\Sym_o(-X)\bigr).\label{eq:Alt B-species}
  \end{align}
  We will return to this example once $\LL$-species have been introduced.
\end{example}

The \emph{derivative} of a species $F$,
denoted $F'$, is defined via \[F'[U] = F[U \sqcup \{\star\}] \] and for
a bijective map
$\sigma: U \to V, \text{ we define } F'[\sigma] = F[\tau]$, where
\begin{equation*}
  \tau(x) = 
  \begin{cases}
    \sigma(x) & \text{if } x \in U, \\
    \star     & \text{if } x=\star.
  \end{cases}
\end{equation*}
In terms of generating series, we have
$F'(x)=\frac{d}{dx} [F(x)]$.

\begin{example}
To illustrate the derivative of a species, we look at the derivative
  of the species of linear orders. We claim that $L' =
  L^2$. Combinatorially, each $L'$-structure is simply an $L$-structure
  preceding $\star$ followed by another $L$-structure. That is, the
  derivative of a linear order just separates the linear order into two
  linearly ordered components. For instance, $41 \star 5263 \mapsto (41,5263)$.
  Consequently, we have
  \[ L'(x) 
    = \frac{d}{dx}\left[\frac{1}{1-x}\right]
    = \frac{1}{(1-x)^2} = L^2(x).
  \]
\end{example}

The last operation we will define for $\BB$-species is
\emph{pointing}. For a species $F$, we define the species $F^{\bullet}$,
called \emph{$F$-pointed}, via
\[
  F^{\bullet}[U] = F[U] \times U.
\]
That is, an \emph{$F^{\bullet}$-structure} on $U$ is a pair $(s,u),$
where $s$ is an $F$-structure on $U$ and $u\in U$ is a distinguished
element that we can think of as being ``pointed at''. The operations of
pointing and differentiation are related by
\[
F^\bullet = X \cdot F'.
\]
Further, we have $|F^{\bullet}[n]| = n |F[n]|$.

Next we look at $\LL$-species. The key differences is that for $\LL$-species, the
underlying sets are totally ordered and morphisms have to respect that order. 

Let $\ell_1 = (U_1, \preceq_1)$ and $\ell_2=(U_2,\preceq_2)$ be two totally 
ordered sets. Their \emph{ordinal sum} is the totally
ordered set $\ell=(U, \preceq)$, denoted by $\ell=\ell_1 \oplus \ell_2$, where
\begin{equation*}
  u \preceq v \;\iff\;
  \begin{cases}
    u \prec_1 v  & \text{if } u,v\in U_1, \\
    u \prec_2 v  &  \text{if } u,v\in U_2, \\
    u \in U_1 \text{ and } v\in U_2  & \text{otherwise}. \\
  \end{cases}
\end{equation*}
In other words, $\ell$ respects $\ell_1$ and $\ell_2$, and all elements of $\ell_1$
are smaller than the elements of $\ell_2$. The totally ordered set obtained
by adding a new minimum element to $\ell$ is denoted
by $1 \oplus \ell$. Similarly, we can append a new maximum element to obtain 
$\ell \oplus 1$.

An \emph{$\LL$-species} is a functor $F:\LL\to\BB$, where $\LL$ is the category
of finite totally ordered sets with order preserving bijections. In other words,
an $\LL$-species is a rule $F$ that associates
\begin{itemize}
\item for each finite totally ordered set $\ell$, a finite set $F[\ell]$;
\item for each order preserving bijection $\sigma: \ell_1 \to \ell_2$, a
  bijection $F[\sigma]: F[\ell_1] \to F[\ell_2]$ such that
  $F[\sigma \circ \tau] = F[\sigma] \circ F[\tau]$ for all order
  preserving bijections $\sigma: \ell_1 \rightarrow \ell_2$,
  $\tau: \ell_2 \rightarrow \ell_3$, and $F[\id_{\ell}]=\id_{F[\ell]}$.
\end{itemize}

Any $\BB$-species $F$ produces an $\LL$-species, also denoted by
$F$, defined by setting
\[
  F[(U, \preceq)] = F[U],
\]
for any totally ordered set $\ell=(U,\preceq)$, where the transport of
structure is obtained by restriction to order preserving bijections. 
All of the $\BB$-species defined earlier have the same name when
interpreted as $\LL$-species. 

Two $\LL$-species $F$ and $G$ are \emph{isomorphic} if there is a family
of bijections
\[
\alpha_{\ell}: F[\ell] \to G[\ell],
\]
for each totally ordered set $\ell$ that commutes with the transports of
structures. 

For an $\LL$-species $F$, the associated generating series
is defined in the same way as for $\BB$-species:
$F(x) = \sum_{n \geq 0} |F[n]| x^n/n!$.
Since there is a unique order preserving 
bijection between any two totally ordered sets with the 
same cardinality we have that, for $\LL$-species, $F=G$ if and only if $F(x)=G(x)$.
In particular, it is possible for two
nonisomorphic $\BB$-species to become isomorphic when looked at
as $\LL$-species. 
For example, as $\LL$-species, $L=\Sym$. In this case,
$f \in \Sym[n]$ naturally becomes the word
$f(1) f(2) \ldots f(n) \in L[n].$

Operations on $\mathbb{B}$-species can be extended to $\LL$-species while new operations 
such as integration,
ordinal product, and convolution also become possible. 
For an $\LL$-species $F$, we define the \emph{derivative} $F'$ via
\[
F'[\ell] = F[1 \oplus \ell]. 
\]
Certainly, one can equivalently append a new maximal element to a totally ordered set 
$\ell$ as opposed to a new minimal element,
hence we have the alternative definition $F'[\ell] = F[\ell\oplus 1]$.
We also define the \emph{integral} of $F$, denoted $\int_0^X F(T) dT$, or more simply $\int F$, by
\[
\left( \int F \right) [\ell] =
\begin{cases}
  \emptyset &  \text{if }\ell = \emptyset, \\
  F[\ell \setminus \{\min(\ell)\}] & \text{if } \ell \neq \emptyset.
\end{cases}
\]
Given $\LL$-species $F$ and $G$, we define the \emph{ordinal product} $F\odot G$ via
\[
(F \odot G)[\ell] = \sum_{\ell=\ell_1 \oplus \ell_2} F[\ell_1] \times G[\ell_2].
\]
In contrast to the product structure, an ordinal product structure
$(F \odot G)[\ell]$ is obtained by splitting $\ell$ into an initial segment
$\ell_1$ and a terminal segment $\ell_2$, where $\ell_1$ has an $F$-structure
and $\ell_2$ has a $G$-structure.
The \emph{convolution product} $F*G$ is given by
\[
F * G = F \odot X \odot G.
\]
Details about the corresponding transport of structures for 
each of these $\LL$-species operations can be found in
the usual book~\cite{holybook}.

As expected, if $F$ and $G$ are $\LL$-species, then we have the following:
\begin{enumerate}[label=(\alph*)]
\item \, $(F+G)(x) = F(x) + G(x)$;
\item \, $(F\cdot G)(x) = F(x)G(x)$;
\item \, $(F \circ G)(x) = F(G(x))$, where $G(0)=0$;
\item \, $\displaystyle F'(x) = \frac{d}{dx} F(x)$;
\item\label{prop:speciesintegral}
 \, $\displaystyle \left( \int_0^X F(T) dT \right)(x) = \int_o^x F(t) dt$;
\item \, $\displaystyle (F * G)(x) = F(x) * G(x)=\int_0^x F(x-t)G(t)dt$.
\end{enumerate}
In addition,
\begin{enumerate}[resume, label=(\alph*)]
\item \, $\displaystyle \frac{d}{dx} [F(x)*G(x)] = F(0) \cdot G(x) + F'(x) * G(x)$,
\end{enumerate}
which is sometimes referred to as the \emph{Leibniz rule}.

\begin{example}[Alternating group as an $\LL$-species]\label{ex:Alt L-species}
  Clearly, we have the $\BB$-species identity $E=\Even{E}+\Odd{E}$,
  which is reflected in the familiar hyperbolic functions identity
  $e^x=\cosh(x)+\sinh(x)$. In a similar vein, consider the combinatorial
  counterpart of the identity $\cosh(x)^2-\sinh(x)^2=1$, namely
  $\Even{E}^2=1+\Odd{E}^2$. This holds as an identity among $\LL$-species,
  but it is false as a $\BB$-species identity. While there are two
  ``unlabeled'' $\Even{E}^2$-structures of size two, namely $(\{\circ,\circ\},\emptyset)$
  and $(\emptyset,\{\circ,\circ\})$, there is a single unlabeled
  $(1+\Odd{E}^2)$-structure of size two, namely $(\{\circ\},\{\circ\})$.
  Continuing with the $\LL$-species setting, we have
  \begin{align*}
    1 &= \Even{E}^2-\Odd{E}^2 \\
      &= (\Even{E}-\Odd{E})(\Even{E}+\Odd{E}) \\
      &= (\Even{E}-\Odd{E})\cdot E,
  \end{align*}
  which when composing with $\Cyc$ yields $1 = (\Sym_e-\Sym_o)\cdot S$.
  The multiplicative inverse of $\Sym$ is the virtual species $1-X$, and hence
  $\Sym_e-\Sym_o = 1-X$. Applying this to expression~\eqref{eq:Alt B-species}
  for the species $\Alt$ in Example~\ref{example: the alternating group},
  we arrive at
  \begin{equation}\label{eq:Alt L-species}
    \Alt = \frac{1}{2}\bigl(\Sym + 1+ X\bigr).
  \end{equation}
\end{example}

Continuing with the topic of permutations, we present
the following proposition.

\begin{proposition}\label{prop: Sym convolution}
  $\Sym = E + E*\Sym^{\bullet}$.
\end{proposition}
\begin{proof}
  Let $n\geq 0$.  We shall construct a bijection from
  $\Sym[n]$ to $(E + E*\Sym^{\bullet})[n]$.  Let a permutation
  $\pi\in\Sym[n]$ be given. If $\pi=12\ldots n$ is the identity
  permutation, then we map $\pi$ to $\{1,2,\ldots,n\}$, the single
  element of $E[n]$. Assume that $\pi$ has at least one point that is
  not fixed.  In this case we will have to map $\pi$ to an
  element of $(E*\Sym^{\bullet})[n]$. Recall that, by definition of
  convolution,
  \[
    E*\Sym^{\bullet} = E\odot X \odot \Sym^{\bullet}.
  \]
  Now, let $k$ be the largest integer such that $1,2,\dots,k$ are fixed
  points in $\pi$. Note that $0\leq k<n-1$. We interpret this maximal
  prefix of fixed points as the set $\{1,2,\dots,k\}$ (an
  $E$-structure). Consider the remaining suffix of $\pi$. Remove $k+1$
  from the cycle it belongs to (this does not alter the number of cycles
  since $k+1$ is not a fixed point) and interpret $k+1$ as a singleton
  (an $X$-structure). From what remains of the suffix we form an
  $\Sym^{\bullet}$-structure by pointing at the element $\pi(k+1)$. Note
  that $\pi(k+1) > k+1$ by the choice of $k$.
\end{proof}

\begin{example}\label{ex: bijection illustration}
  Below is an example of the bijection in the proof of
  Proposition~\ref{prop: Sym convolution}.
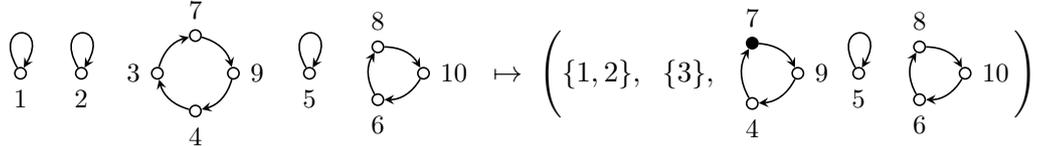
  \[
  \begin{tikzpicture}[
    dot/.style = {circle, draw, inner sep=1.5pt, node contents={}, label=#1},
    every loop/.style={min distance=7mm,in=60,out=120,looseness=10},
    semithick,
    baseline=-3pt
    ]
    \begin{scope}[font = {\small}]
      \node (c1) at (-0.3,0) [dot=below:$1$];
      \path (c1) edge[loop above, -stealth] (c1);
      \node (c2) at (0.5,0) [dot=below:$2$];
      \path (c2) edge[loop above, -stealth] (c2);
      \node (c3-3) at (1.5,0) [dot=left:$3$];
      \node (c3-7) at (2,0.5) [dot=above:$7$];
      \node (c3-9) at (2.5,0) [dot=right:$9$];
      \node (c3-4) at (2,-0.5) [dot=below:$4$];
      \path (c3-3) edge[-stealth, bend left] (c3-7);
      \path (c3-7) edge[-stealth, bend left] (c3-9);
      \path (c3-9) edge[-stealth, bend left] (c3-4);
      \path (c3-4) edge[-stealth, bend left] (c3-3);
      \node (c4) at (3.5,0) [dot=below:$5$];
      \path (c4) edge[loop above, -stealth] (c4);
      \node (c5-6) at (4.4, -0.35) [dot=below:$6$];
      \node (c5-8) at (4.4, 0.35) [dot=above:$8$];
      \node (c5-10) at (5,0) [dot=right:$10$];
      \path (c5-6) edge[bend left, -stealth] (c5-8);
      \path (c5-8) edge[bend left, -stealth] (c5-10);
      \path (c5-10) edge[bend left, -stealth] (c5-6);
    \end{scope}
  \end{tikzpicture}
  \;\mapsto\;
  \Biggl(
  \{1,2\}, \;\; \{3\}, \;\;
  \begin{tikzpicture}[
    dot/.style = {circle, draw, inner sep=1.5pt, node contents={}, label=#1},
    every loop/.style={min distance=7mm,in=60,out=120,looseness=10},
    semithick,
    baseline=-3pt
    ]
    \begin{scope}[font = {\small}]
      \node[fill] (c3-7) at (2,0.4) [dot=above:$7$];
      \node (c3-9) at (2.6,0) [dot=right:$9$];
      \node (c3-4) at (2,-0.4) [dot=below:$4$];
      \path (c3-7) edge[-stealth, bend left] (c3-9);
      \path (c3-9) edge[-stealth, bend left] (c3-4);
      \path (c3-4) edge[-stealth, bend left] (c3-7);
      \node (c4) at (3.4,0) [dot=below:$5$];
      \path (c4) edge[loop above, -stealth] (c4);
      \node (c5-6) at (4.2, -0.35) [dot=below:$6$];
      \node (c5-8) at (4.2, 0.35) [dot=above:$8$];
      \node (c5-10) at (4.8, 0) [dot=right:$10$];
      \path (c5-6) edge[bend left, -stealth] (c5-8);
      \path (c5-8) edge[bend left, -stealth] (c5-10);
      \path (c5-10) edge[bend left, -stealth] (c5-6);
    \end{scope}
  \end{tikzpicture}\!
  \Biggr)
  \]
  its \emph{functional digraph}:
  the directed graph with vertex
  set $[n]$ and arcs $i\to \pi(i)$. The vertex pointed at is black.
 \end{example}

Let $\AltC$ denote the subspecies of $\Sym$ consisting of the odd
permutations, which is the complement of the alternating group in the
sense that $\Sym=\Alt+\AltC$. We have the following corollary of 
Proposition~\ref{prop: Sym convolution}, which will turn out to be useful when 
proving Theorem~\ref{prop: species eq for Cay(212)}.

\begin{corollary}\label{cor:X}
  $\AltC = E*\Alt^{\bullet}$.
\end{corollary}

\begin{proof}
  Let $\pi$ be an odd permutation. Consider the image of $\pi$ under the
  bijection from the proof of Proposition~\ref{prop: Sym convolution}.
  Removing a maximal prefix $1$, $2$, \dots, $k$ of fixed points from
  $\pi$ does not alter the parity of $\pi$. On the other hand, deleting
  $k+1$ from its cycle shortens that cycle by one and reverses its parity.
  It follows that the restriction of this bijection to odd permutations
  proves our corollary. See Example~\ref{ex: bijection illustration},
  above, for an example of this bijection.
\end{proof}

\begin{example}[$231$-avoiding permutations]\label{ex:catalan species}
  A permutation or linear order $w=w_1w_2\ldots w_n$ on $[n]$
  contains $231$ (as a pattern) if it contains a subsequence $w_iw_jw_k$
  that is order isomorphic to $231$; that is, if
  $w_k<w_i<w_j$. Otherwise, $w$ is $231$-avoiding.
  It is well known that the number of $231$-avoiding permutations of $[n]$ is given
  by the $n$th Catalan number, $\binom{2n}{n}/(n+1)$. As an illustration
  of the species approach adopted in this article we shall now give a
  somewhat unorthodox proof of this fact.
  Let $F=\Sym(231)$ be the $\LL$-species of $231$-avoiding permutations, 
  which is a subspecies of $\Sym$.
  Let $\ell$ be a nonempty totally ordered set and let
  $w\in\Sym[\ell]$. Then $w$ belongs to $F[\ell]$ if and only if there
  are totally ordered sets $\ell_1$ and $\ell_2$ such that
  $\ell=\ell_1\oplus\ell_2\oplus 1$, and $w$ can be factored as $w=umv$,
  where $u\in F[\ell_1]$, $v\in F[\ell_2]$, and $m=\max(\ell)$. In other
  words, $F$ is characterized be the following equation:
  \[
    F=1+F\odot X \odot F = 1+F*F.
  \]
  Our goal is to show that $a_n=|F[n]|$ is the $n$th Catalan
  number. Note that $a_n = F^{(n)}(0)$, where $F^{(n)}$ is the $n$th
  derivative of $F$. Now, by the Leibniz rule,
  \[
    F'(x) = (F'*F)(x) + a_0F(x).
  \]
  Differentiating this expression and, again, using the Leibniz rule,
  we obtain
  \[
    F''(x) = (F''*F)(x) + a_0F'(x)+a_1F(x).
  \]
  Continuing this way we find that
  \[
    F^{(n+1)}(x) = (F^{(n+1)}*F)(x) + \sum_{i=0}^na_iF^{(n-i)}(x).
  \]
  Note that $(F^{(n+1)}*F)(0)=0$ and hence, by identifying coefficients in the
  above identity when $x=0$, we obtain $a_0=1$ and
  \[
    a_{n+1}=\sum_{i=0}^na_ia_{n-i},
  \]
  which is the usual recurrence for the Catalan numbers. We will present a
  kindred recurrence for the number of $231$-avoiding Cayley permutations in
  Proposition~\ref{prop: Cay(231) recurrence}.
\end{example}

\section{Cayley permutations}\label{sec:cayleyperms}

In the introduction we defined a Cayley permutation as a word $w$ of
positive integers such that if a letter $b$ appears in $w$ then all
positive integers $a<b$ also appear in~$w$. Expressed differently, a Cayley
permutation is a function $w:[n]\to [n]$ such that $\Img(w) = [k]$ for
some $k\leq n$. To define the $\BB$-species of Cayley permutations we
have to generalize this a bit by allowing the domain to be any $n$ element
set. Thus, a
\emph{Cayley permutation} on a finite set $U$ is a function
$w: U \rightarrow [n]$ such that $|U| = n$ and $\Img(w) = [k]$ for some $k \leq n$.
And we let $\Cay$ be the $\BB$-species with structures
\[ 
\Cay[U] \,=\,
  \bigl\{
    w\in [n]^U \!:\, \Img(w) = [k] \text{ for some } k \leq n
  \bigr\},
\]
together with the transport of structure along a bijection
$\sigma: U \to V$ defined via 
\[
\Cay[\sigma](w)=w \circ \sigma^{-1}.
\]
For $w \in \Cay[n]$, we utilize one-line notation and write
$w = w_1w_2\ldots w_n$, where $w_i = w(i)$. In this case, we say that $w$ is of
\emph{length} $n$. 

For example,
\begin{itemize}
\item $\Cay[1] = \{1\}$;
\item $\Cay[2] = \{11, 12, 21\}$;
\item $\Cay[3] = \{111, 112, 121, 122, 123, 132, 211, 212, 213, 221, 231, 312, 321\}$.
\end{itemize}

For a finite set $U$, we define
$\alpha_U: \Cay[U] \to \Bal[U]$ via
\[
  \alpha_U(w) = B_1B_2\ldots B_k,
  \text{ where } k=|\Img(w)| \text{ and } B_i = w^{-1}(\{i\}).
\]
This map is clearly bijective and hence
$|\Bal[n]| = |\Cay[n]|$. It also follows
(see Example~\ref{example of ballots}) that
$\Cay(x) = \Bal(x) = 1/(2-e^x)$. We shall see in
Proposition~\ref{prop: as species, cay=bal} below that, via the bijections
$\alpha_U$, the species $\Cay$ and $\Bal$ are in fact isomorphic.

\begin{example}\label{ex: transport cay and bal}
We wish to illustrate how the bijections $\alpha_U$
respect the transport of structure induced by $\Cay$ and $\Bal$, a fact we
will prove formally in Proposition~\ref{prop: as species, cay=bal}.
Let $U=[8]$ and $V=\{\wking,\bking,\wrook,\brook,\wknight,\bknight,\wpawn,\bpawn\}$.
Define a bijection $\sigma:U\longrightarrow V$ by
$$
\sigma={\setlength\arraycolsep{1pt}\begin{pmatrix}
    1 & 2 & 3 & 4 & 5 & 6 & 7 & 8\\
    \wking & \wrook & \wknight & \wpawn & \bking & \brook & \bknight & \bpawn
    \end{pmatrix}}.
$$
That is, $\sigma(1)=\wking, \sigma(2)=\wrook, \sigma(3)=\wknight$, etc.
The action of the maps $\Bal[\sigma]\circ\alpha_U$ and $\alpha_V\circ\Cay[\sigma]$
on the $\Cay[U]$-structure $w=31342224$ is illustrated by the
commutative diagram in Figure~\ref{fig: transport cay and bal}.
\end{example}

\begin{figure}
\centering
\begin{tikzpicture}[xscale=1.15, yscale=1.05]
\node (a) at (0,0)
{$\setlength\arraycolsep{1pt}\begin{pmatrix}
\wking & \wrook & \wknight & \wpawn & \bking & \brook & \bknight & \bpawn\\
3 & 1 & 3 & 4 & 2 & 2 & 2 & 4
\end{pmatrix}$};
\node (b) at (0,2) {$31342224$};
\node (c) at (6,0) {$(\{\wrook\},\{\bking,\brook,\bknight\},\{\wking,\wknight\},\{\wpawn,\bpawn\})$};
\node (d) at (6,2) {$(\{2\},\{5,6,7\},\{1,3\},\{4,8\})$};
\draw [black,-stealth, very thick] (b) to node[midway, left, magenta]
{$\Cay[\sigma]$} (a);
\draw [black,-stealth, very thick] (b) to node[midway, above, magenta]
{$\alpha_U$} (d);
\draw [black,-stealth, very thick] (d) to node[midway, right, magenta]
{$\Bal[\sigma]$} (c);
\draw [black,-stealth, very thick] (a) to node[midway, below, magenta]
{$\alpha_V$} (c);
\end{tikzpicture}
\caption{An instance of the compatibility between the transport of structures for $\Cay$ 
and $\Bal$ described in Example~\ref{ex: transport cay and bal}.}
\label{fig: transport cay and bal} 
\end{figure}
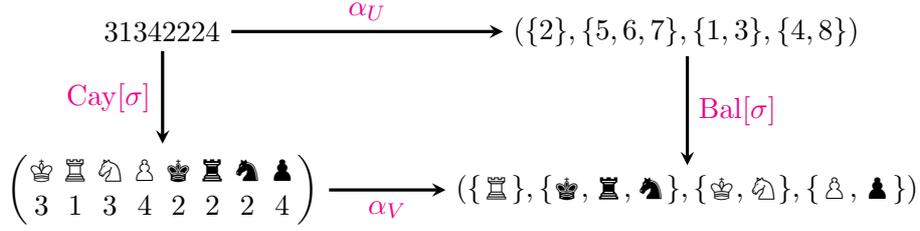

\begin{proposition}\label{prop: as species, cay=bal}
  As $\BB$-species, $\Cay=\Bal$.
\end{proposition}

\begin{proof}
Let $\sigma: U \to V$ be a bijection between finite sets $U$ and $V$. We
aim to show that Figure~\ref{fig: isomorphic as species diagram}
commutes for species $\Cay$ and $\Bal$. Let $w \in \Cay[U]$ with
$\Img(w)=[k]$. We see that
\begin{align*}
   \bigl(\alpha_V \circ \Cay\bigr)[\sigma] (w)
  &= \alpha_V( \Cay[\sigma] (w)) \\
  &= \alpha_V(w \circ \sigma^{-1}) \\
  &= (w\circ \sigma^{-1})^{-1}(\{1\})\, \dots\, (w\circ \sigma^{-1})^{-1}(\{k\}) \\
  &= \Bal[\sigma](w^{-1}(\{1\})) \, \dots\, \Bal[\sigma](w^{-1}(\{k\})) \\
  &= \Bal[\sigma]( \alpha_U (w)) \\
  &= \bigl(\Bal[\sigma]\circ \alpha_U\bigr)(w). \qedhere
\end{align*}
\end{proof}

\section{Pattern avoidance in Cayley permutations}\label{sec:pattsonCayley}

In this section, we will introduce pattern avoidance in the context of
Cayley permutations and ballots. In order to sensibly define pattern
avoidance for Cayley permutations we must restrict the domain to a
totally ordered set, which allows us to write Cayley permutations 
in one-line notation. One consequence of this is that all species of 
pattern-avoiding Cayley permutations are $\LL$-species. Furthermore, since 
there is a unique order preserving
bijection between any pair of finite totally ordered sets of the same
cardinality we may, when convenient, assume that the underlying total
order is $[n]$ with the standard order.

Consider the Cayley permutations
$w = w_1 w_2 \ldots w_n \in \Cay[n] $ and
$p = p_1 p_2 \ldots p_k \in \Cay[k]$. We say that $w$
\emph{contains} $p$ if there exists a subsequence
$w_{i_1}w_{i_2} \ldots w_{i_k}$ with $i_1 \leq i_2 \leq \ldots \leq i_k$
that is order isomorphic to $p$
(i.e., for every $a<b$ we have $w_{i_a}<w_{i_b}\Leftrightarrow p_a<p_b$ and
$w_{i_a}=w_{i_b} \Leftrightarrow p_a=p_b$).
If $w$ does not contain a pattern $p$, then we say that
$w$ \emph{avoids} $p$.

Let $\Cay(p)$ be the $\LL$-species of $p$-avoiding Cayley permutations.
The corresponding structures are
\[
\Cay(p)[n] = \{ w\in \Cay[n] : w \text{ avoids } p \}.
\]
The transport of structure is inherited from $\Cay$.
We also define $\Bal(p)[n]$ to be the
image of $\Cay(p)[n]$ under the natural bijection between
Cayley permutations and ballots described earlier.
It follows that $\Cay(p) = \Bal(p)$ as $\LL$-species.

\begin{example}\label{ex:1^k}
Consider $\Cay(1^k)$ for $k\geq 2$, where 
$1^k = 11 \ldots 1$ consists of $k$ copies of $1$.
A $1^k$-avoiding ballot has blocks of sizes at most $k-1$. That is,
blocks may be any size between $1$ and $k-1$, and hence
\[
  \Cay(1^k)=\Bal(1^k) = L(E_1 + E_2 + \cdots + E_{k-1}).
\]
It follows that
\begin{align*}
\Cay(1^k)(x)
  &= L(E_1 + \cdots + E_{k-1})(x) \\
  &= L(E_1(x) + \cdots + E_{k-1}(x))
  = \left(1 - \sum_{i=1}^{k-1} \frac{x^i}{i!}\right)^{-1}.
\end{align*}
\end{example}

We now extend the idea of Wilf equivalence~\cite{permpatternsbasicdefinitions}
to Cayley permutations.
We say that two patterns $p$ and $q$ are \emph{(Wilf) Cayley-equivalent}
if $|\Cay(p)[n]|=|\Cay(q)[n]|$ for all $n$; equivalently, if
$\Cay(p)=\Cay(q)$ as $\LL$-species.

The standard reverse and complement maps on permutations are easily generalized to
Cayley permutations.
Given a Cayley permutation $w=w_1\ldots w_n$, the \emph{reverse map}
is defined via
\[
r(w_1\ldots w_n) = w_n\ldots w_1
\]
and the \emph{complement map} is given by
\[
c(w_1 \ldots w_n) = (\max(w)+1-w_1)\ldots (\max(w)+1-w_n),
\]
where $\max(w) = \max\{w_i: i\in[n]\}$.
Certainly, the reverse and complement maps are bijections on $\Cay[n]$.
This implies that $p$, $r(p)$, $c(p)$, and $(r\circ c)(p)$ are all
Cayley-equivalent for a pattern~$p$. 
The class generated by reverse
and complement of a single pattern~$p$ is the \emph{symmetry class} of~$p$.
Note that two patterns may be Cayley-equivalent without being in the
same symmetry class.
Note also that there is no counterpart to the inverse map on $\Cay[n]$,
a reason for which Cayley-equivalence does not correspond to the
usual Wilf equivalence on permutations.
For instance, let $p=1342$ and $q=1423$. Then~$p$ and~$q$
are Wilf equivalent on permutations since $p=q^{-1}$.
On the other hand, $p$ and~$q$ are not Cayley-equivalent:
$$
|\Cay(1342)[7]|=33712\neq 33710=|\Cay(1423)[7]|.
$$
We will further explore this and other notions of equivalence
in Section~\ref{section: wilfeq}.

\section{Patterns of length two}\label{sec: length two}

In this section, we will investigate Cayley permutations avoiding
patterns of length two and we start with the pattern $11$. A Cayley
permutation avoiding $11$ is simply a permutation. Or, in terms of
ballots, a $\Bal(11)$-structure is a ballot with singleton blocks
only. We also note that the pattern $11$ is a special case of
the pattern $1^k$ studied in Example~\ref{ex:1^k}, and that
$\Cay(11) = L(E_1) = L$ and $\Cay(11)(x)=1/(1-x)$.

Next we consider the species $\Cay(21)=\Bal(21)$. Note that one could easily obtain the 
counting formula and corresponding generating series for $\Cay(12)$, 
but we have opted for a species-first approach, in part to illustrate the utility of 
combinatorial species.
A nonempty ballot on
$[n]$ avoids $21$ precisely when the element of its first block form an
initial segment $[k]$ of $[n]$ and the remaining blocks form a ballot
that also avoids $21$. Letting $F=\Bal(21)$ this characterization is
captured by the equation
\[
  F = 1 + E_+ \odot F = 1 + E * F.
\]
Using induction it immediately follows that
\begin{equation}\label{eq:Bal(21)}
  F = 1 + 1*\sum_{k\geq 1} E^{*k},
\end{equation}
in which $E^{*k} = E * \cdots * E$ is the convolution of $k$ copies of $E$.

\begin{lemma}
  For any $k\geq 1$, we have $E^{*k} = E_{k-1}E$.
\end{lemma}
\begin{proof}
  For ease of notation let $k=3$. 
  By definition of convolution,
  \[
    E^{*3} = E \odot X \odot E \odot X\odot  E.
  \]
  Consequently, an $E^{*3}$-structure on a totally ordered set $\ell$ is
  a 5-tuple
  \[
    \gamma = (\ell_1,x_1,\ell_2,x_2,\ell_3)
    \,\text{ such that }\,
    \ell = \ell_1\oplus\{x_1\}\oplus\ell_2\oplus\{x_2\}\oplus\ell_3.
  \]
  Clearly, $\gamma$ is uniquely determined by the pair
  $\bigl(\{x_1,x_2\},\ell_1\cup\ell_2\cup\ell_3\bigr)$,
  and this sets up a bijection proving $E^{*3} = E_{2}E$. The general
  case is analogous.
\end{proof}

Continuing with derivation~\eqref{eq:Bal(21)} prior to this lemma, we find that
\begin{align*}
  F
  &= 1 + 1*\sum_{k\geq 1} E^{*k} \\
  &= 1+ 1*E\sum_{k\geq 1} E_{k-1} = 1 + 1*E^2 = 1 + \int\!\! E^2.
\end{align*}

\begin{proposition}\label{prop: Cay(21) species}
  $\displaystyle\Cay(21) = 1 + \int\! E^2 = \Even{E} \cdot E$.
\end{proposition}
\begin{proof}
  Let $F = 1 + \int\! E^2$ and $G=\Even{E} \cdot E$. We have
  already established that $\Cay(21)=F$. To show that
  $F=G$ we note that $F(0)=1=G(0)$, $F'=E^2$, and
  \[
    G'
    = \Even{E}' E + \Even{E} E'
    = \Odd{E} E + \Even{E} E = E^2.
  \]
  Alternatively, we may prove $F=G$ using an explicit bijection:
  Let $\ell$ be a finite totally ordered set. Define
  $\alpha_{\ell}: F[\ell]\to G[\ell]$ as follows. For $\ell=\emptyset$ let
  $\alpha_{\ell}(\emptyset) = (\emptyset,\emptyset)$. Assume $\ell$ is
  nonempty and let $x = \min(\ell)$. An $F$-structure on $\ell$ is
  a pair of disjoint sets $(S,T)$ such that
  $S \cup T = \ell \setminus \{x\}$ and we define
  \[
    \alpha_{\ell}(S,T) =
    \begin{cases}
      (S, T\cup \{x\}) & \text{if $|S|$ is even}, \\
      (S\cup \{x\}, T) & \text{if $|S|$ is odd}.
    \end{cases}
  \]
  It is easy to verify that $\alpha_{\ell}$ indeed is a bijection.
\end{proof}

\begin{corollary}\label{cor:egf of cay(12)}
  $\Cay(21)(x) = \cosh(x)e^x = \frac{1}{2}(e^{2x} + 1)$.
\end{corollary}

On extracting the coefficient of $x^n/n!$ in $\frac{1}{2}(e^{2x}+1)$ we
find that
\[
  \bigl|\Cay(21)[\emptyset]\bigr|=1\;\text{ and }\;
  \bigl|\Cay(21)[n]\bigr|=2^{n-1}\;\text{ for $n \geq 1$.}
\]
We could of course have derived this more directly by noting that any
$21$-avoiding Cayley permutation is weakly increasing and hence
consists of a segment of $1$s, followed by a segment of $2$s, followed
by a segment of $3$s, etc. Thus there are as many $21$-avoiding Cayley
permutations as there are integer compositions of $n$.

\section{Patterns of length three}\label{sec: length three}

Next we will study Cayley permutations that avoid patterns
of length three.
It follows from Example~\ref{ex:1^k} that
$\Cay(111) = L(E_1 + E_2)$ and
$\Cay(111)(x)= 2/(2-2x-x^2)$.
Using a partial fraction decomposition of $2/(2-2x-x^2)$ we may also
derive the following explicit counting formula for all $n\geq 0$:
  \[
    |\Cay(111)[n]|
    = n!\cdot\frac{(1+\sqrt{3})^{n+1} - (1-\sqrt{3})^{n+1}}{2^{n+1}\sqrt{3}}.
  \]

Although $112$ and $212$ are Cayley-equivalent as a consequence of the results
proved by Jel\'inek and Mansour on $k$-ary words~\cite{k-ary} (see also
Section~\ref{section: wilfeq}), we will prove it here independently by
showing that $\Cay(112) = \Cay(212)$ as $\LL$-species. We will first
provide a species equation for $\Cay(112)$.

\begin{proposition}\label{prop: species equation Cay(112)}
$\Cay(112)' = L \cdot \Cay(112) + L \cdot \Cay_+(112)$. 
\end{proposition}

\begin{proof}
  Let $n\geq 0$.  We shall verify the result in terms of ballots by
  exhibiting a bijection from $\Bal(112)'[n]=\Bal(112)[n+1]$ to
  $(L \cdot \Bal(112) + L \cdot \Bal_+(112))[n]$.  Consider a ballot
  $\varpi=B_1B_2\ldots B_k$ in $\Bal(112)[n+1]$.  For any $j\in [n+1]$,
  each block preceding the block containing $j$ contains at most one
  value less than $j$. In particular, if the maximum value $n+1$ belongs
  to the $i$th block, then all preceding blocks must
  be singletons, say $B_1=\{a_1\}$, $B_2=\{a_2\}$, \dots,
  $B_{i-1}=\{a_{i-1}\}$. Note that we may view $a_1a_2\dots a_{i-1}$ as
  an $L$-structure. Let $\tilde{B}_i = B_i\setminus \{n+1\}$.  It is easy
  to check that
  \[
    \varpi \,\longmapsto\,
    \begin{cases}
      \bigl(a_1a_2\dots a_{i-1}, B_{i+1}\dots B_k \bigr) & \text{if $\tilde{B}_i=\emptyset$},\\
      \bigl(a_1a_2\dots a_{i-1}, \tilde{B}_iB_{i+1}\dots B_k \bigr) & \text{otherwise}
    \end{cases}
  \]
  is the sought bijection.
\end{proof}

\begin{theorem}\label{thm:species 112}
  $\Cay(112) = \Alt'$.
\end{theorem}
\begin{proof}
  Clearly, $\Cay(112)[\emptyset] = \Alt'[\emptyset]$ and hence it suffices
  to show that $F=\Alt'$ satisfies the differential equation
  $F' = L\cdot (F+F_+)=L\cdot (2F-1)$ from Proposition~\ref{prop: species equation Cay(112)}. 
  Recall from equation~\eqref{eq:Alt L-species} that
  $\Alt = (\Sym + 1 + X)/2 = (L + 1 + X)/2$ and hence $\Alt' = (1+L')/2$.
  Finally, using $L'=L^2$, we get
  $F'  = \Alt'' = L^3 = L\cdot (2F-1)$.
\end{proof}

By the preceding theorem we have $\Cay(112)=\Alt'=(1+L^2)/2$ and since
$L(x)=1/(1-x)$, we arrive at the following result.

\begin{corollary}\label{prop:egf cay(112)}
  $\displaystyle \Cay(112)(x)
  = \frac{1}{2}\left(1+\frac{1}{(1-x)^2}\right)
  = \frac{x^2-2x+2}{2(x-1)^2}$.
\end{corollary}

We now transition to $212$-avoiding Cayley permutations.

\begin{proposition}\label{prop: species eq for Cay(212)}
  $\Cay(212) = 1 + E * \Cay(212) + E * \Cay(212)^{\bullet}$.
\end{proposition}

\begin{proof}
  Let $n\geq 0$.
  For any species $F$ we have $E*F = E\odot X \odot F = E_+ \odot F$. It
  thus suffices to provide a bijection from
  $\Cay(212)[n]$ to
  \[
    \bigl(1 \,+\, E_+\odot\Cay(212)  \,+\, E_+\odot\Cay(212)^{\bullet}\bigr)[n]
  \]
  For $n=0$ we map the empty Cayley permutation to the empty set. Assume
  $n>0$ and $w\in\Cay(212)[n]$. Observe
  that the occurrences of the largest value $m=\max(w)$ in $w$ must
  occur as a single contiguous string. That is, we can write
  \[
    w \,=\, w_1\ldots w_i m \ldots m w_{i+d+1} \ldots w_n,
  \]
  where $d$ is the number of occurrences of $m$ in $w$ and both
  $w_1 \ldots w_i$ and $w_{i+d+1} \ldots w_n$ are 212-avoiding Cayley
  permutations in their own right. The size of the $E_+$-structure will
  be $d$ and its role is simply to keep track of the number of occurrences of
  $m$. We consider two cases according to whether or
  not the contiguous string $m \ldots m$ occurs on the far right. In the
  first case,
  \begin{align*}
    w &\;\longmapsto\;
    \bigl(\,\{1,\dots,d\},\, w_1\ldots w_i\,\bigr). \\
    \intertext{In the second case, we can identify a distinguished position
    immediately to the right of the contiguous string $m \ldots m$:}
    w &\;\longmapsto\;
    \bigl(\,\{1,\dots,d\},\, w_1\ldots w_i w_j^{\bullet}\ldots w_n\,\bigr).\qedhere
  \end{align*}
\end{proof}

Having established an equation characterizing $\Cay(212)$ we shall show
in Theorem~\ref{thm:species 212} that $\Cay(212) = \Alt'$
satisfies that equation, but first a simple lemma.

Recall that $\AltC=\Sym-\Alt$ denotes the species of odd permutations.

\begin{lemma}\label{lemma:odd perms}
  $\Alt = 1+X+\AltC$.
\end{lemma}
\begin{proof}
  This can be proved bijectively by fixing a transposition, say $\tau$, and
  observing that multiplication by $\tau$ reverses the
  parity. Alternatively, from equation~\eqref{eq:Alt L-species} we see
  that $2\Alt = 1 + X + \Sym = 1 + X + \Alt + \AltC$ and hence
  $\Alt = 1 + X + \AltC$.
\end{proof}

\begin{theorem}\label{thm:species 212}
  $\Cay(212) = \Alt'$.
\end{theorem}
\begin{proof}
  If we let $F=\Cay(212)$, then Proposition~\ref{prop: species eq for Cay(212)} reads
  \begin{equation}\label{eq:Cay(212)}
    F = 1 + E * F + E * F^{\bullet}.
  \end{equation}
  We shall prove that $F=\Alt'$ satisfies this equation. Starting from
  the right-hand side we have
  \begin{align*}
    1 + E*\Alt' + E*(\Alt')^{\bullet}
    &= 1 + E*(\Alt' + X\Alt'')\\
    &= 1 + E*(X\Alt')' \\
    &= 1 + E*(\Alt^{\bullet})' \\
    &= (1 + X + E*\Alt^{\bullet})' && \text{(by the Leibniz rule)} \\
    &= (1 + X + \AltC)' && \text{(by Corollary~\ref{cor:X})} \\
    &= \Alt' && \text{(by Lemma~\ref{lemma:odd perms})},
  \end{align*}
  which concludes the proof.
\end{proof}

\begin{corollary}\label{prop:egf cay(212)}
  $\displaystyle \Cay(212)(x)
  = \displaystyle\frac{1}{2}\left(1+\frac{1}{(1-x)^2}\right)
  = \frac{x^2-2x+2}{2(x-1)^2}$.
\end{corollary}

For $n\geq 1$, the number of alternating permutations of $[n]$ is $n!/2$ and since
$\Cay(112) = \Cay(212)=\Alt'$, by Theorems~\ref{thm:species 112} and~\ref{thm:species 212},
we immediately get the following counting formula for
$\Cay(112)$ and $\Cay(212)$.

\begin{corollary}\label{cor: counting 212}
  We have $|\Cay(112)[\emptyset]|=|\Cay(212)[\emptyset]|=1$ and, for $n \geq 1$,
  \[
    \bigl|\Cay(112)[n]\bigr| = \bigl|\Cay(212)[n]\bigr| =  \frac{1}{2}(n+1)!
  \]
\end{corollary}

Using the reverse and complement maps, we have the $\LL$-species identities
$$
\Cay(112)=\Cay(221)=\Cay(211)=\Cay(122)
\quad\text{and}\quad
\Cay(212)=\Cay(121).
$$
However, the species $\Cay(112)$ and $\Cay(212)$ are the 
same according to Theorems~\ref{thm:species 112} and~\ref{thm:species 212},
which implies that all six of the patterns are Cayley-equivalent.

Birmajer, Gil, Kenepp, and Weiner~\cite{Birmajer-etal} counted weak
orderings subject to various ``stopping conditions''. In particular,
they showed that the number of weak orderings of size $n\ge 1$ subject to
$x_{i_1}<x_{i_2}=x_{i+3}$ is $(n+1)!/2$. Translating this to our
terminology, they showed that $|\Cay(122)[n]|=(n+1)!/2$.
Now, the
reverse-complement of $122$ is $112$, and hence the formula
$|\Cay(112)[n]|=(n+1)!/2$ of Corollary~\ref{cor: counting 212} was
already known.

Let us now consider the species $\Cay(231)$. We wish to provide a geometric
decomposition of $231$-avoiding Cayley permutations that has the same flavor
as the well-known recursive description of $\Sym(231)$. Our technique is
similar to the one used in Proposition~\ref{prop: species eq for Cay(212)}
and it leads to an integral equation for $\Cay(231)$.

\begin{proposition}\label{prop: equation for cay231}
  $\Cay(231) = 1 + X + (4\Cay(231) - E)*\Cay(231)_+$.
\end{proposition}
\begin{proof}
Any nonempty $w=w_1w_2\dots w_n\in\Cay(231)[n]$ factors as
$$
w \,=\, u \max(w) v,
$$
where
$w_i=\max(w)$ is the leftmost copy of $\max(w)$ in~$w$,
$u=w_1\ldots w_{i-1}$ is the prefix of $w$ preceding $w_i$, and
$v=w_{i+1}\ldots w_n$ is the suffix of $w$ succeeding $w_i$.
Note that
$\max(u)<\max(w)$ by our choice of $i\in [n]$. Note also that $\max(u)\le \min(v)$, or else
the triple $\max(u),\max(w),\min(v)$ would form an occurrence of $231$ in $w$.
Further, both~$u$ and~$v$ are (order isomorphic to) $231$-avoiding Cayley
permutations. Now, any $w\in\Cay(231)[n]$ falls into exactly one of the
following cases:
\begin{enumerate}
\item $w=\epsilon$, the empty Cayley permutation.
\item $v=\epsilon$. That is, there is only one copy of $\max(w)$ and it is
the last entry of~$w$. This includes the single letter Cayley permutation
$w=1$, where $u=v=\epsilon$.
\item $u=\epsilon$ and $v\neq \epsilon$. In this case, either
$\max(w)=\max(v)$ or $\max(w)=\max(v)+1$.
\item $u\neq\epsilon$ and $v\neq\epsilon$. Here, there are four sub-cases
to be considered, each giving rise to distinct Cayley permutations.
In the first three, $u$ and~$v$ are allowed to be (order isomorphic to)
any $231$-avoiding Cayley permutations, while the fourth case
needs some additional care.
\begin{enumerate}
\item $\min(v)=\max(u)+1$ and $\max(w)=\max(v)+1$;
\item $\min(v)=\max(u)+1$ and $\max(w)=\max(v)$;
\item $\min(v)=\max(u)$ and $\max(w)=\max(v)+1$;
\item $\min(v)=\max(u)$ and $\max(w)=\max(v)$; here, the case where
$\max(v)=\min(v)$ is forbidden since otherwise we would have
$\max(u)=\max(w)$, which is impossible.
\end{enumerate}
\end{enumerate}
Let $F=\Cay(231)$. The above case analysis leads to the following equation
\begin{equation}\label{eq: Cay(231) species}
F \;=\;
\underbrace{1}_1\;+\;
\overbrace{\int F}^2\;+\;
\overbrace{2\int F_+}^3\;+\;
\!\!\underbrace{3F_+*F_+}_{4\text{(a)}+4\text{(b)}+4\text{(c)}}\!\!+\;
\underbrace{F_+*(F-E)}_{4\text{(d)}},
\end{equation}
where we have annotated the terms with the corresponding cases.  Each
term can be straightforwardly derived using the techniques developed in this
article.  We omit most of the details and only consider the last
term, corresponding to case 4(d):
\[
  F_+*(F-E) = F_+\odot X\odot (F-E).
\]
Consider $w = u \max(w) v\in \Cay(231)[n]$ as above. In case 4(d), $u$
is nonempty, but otherwise unconstrained, which gives the $F_+$
factor. The singleton $\max(w)$ corresponds to the $X$ factor.  Further,
$v$ is nonempty and $\max(v)>\min(v)$, resulting in the $F-E$ factor. The reason for
subtracting $E$ from $F$ is that there is exactly one Cayley permutation
of fixed length and fixed underlying alphabet, where $\max(v)=\min(v)$,
namely the constant sequence. Finally, the total contribution
of 4(d) is obtained as the ordinal product $F_+\odot X\odot (F-E)$ of
the three factors, and on simplifying \eqref{eq: Cay(231) species} we obtain
the claimed equation.
\end{proof}

Recall that in Example~\ref{ex:catalan species} we explored the species
$F=\Sym(231)$ of $231$-avoiding permutations. It satisfies $F = 1 + F*F$
and by repeated differentiation we derived a familiar recursion for the
Catalan numbers. We can apply the same technique to
the species $F=\Cay(231)$ of $231$-avoiding Cayley permutations. Our
starting point is the equation $F = 1 + X + (4F-E)*F_+$ from the
preceding proposition. By repeated differentiation, and use of the
Leibniz rule, we find that
\[
  F^{(n+1)}(x) =
  F^{n+1}(x)*\bigl(4F(x)-E(x)\bigr) + \sum_{i=0}^{n-1}\bigl(4F^{i}(x)-E(x)\bigr)a_{n-i},
\]
where $a_n=F^n(0)=|\Cay(231)[n]|$ is the number of $231$-avoiding Cayley
permutations on $[n]$. Through identifying coefficients in the above identity
when $x=0$, we obtain the recurrence relation presented below.

\begin{proposition}\label{prop: Cay(231) recurrence}
  Let $a_n=|\Cay(231)[n]|$ be the number of $231$-avoiding Cayley
  permutations. Then $a_0=a_1=1$ and, for $n\geq 1$, we have
  \[
    a_{n+1}= \sum_{i=0}^{n-1}(4a_i-1)a_{n-i}.
  \]
\end{proposition}


In 1985 Simion and Schmidt~\cite{simionschmidt} gave a
bijection between $\Sym(123)[n]$ and $\Sym(132)[n]$.
It is considered a classic result
in the permutation patterns literature. Simion and Schmidt presented their
bijection as an algorithm, but it can also be stated in
terms of equivalence classes induced by positions and values of
left-to-right minima as follows~\cite{ClKi2008}. Given a permutation
$w=w_1\ldots w_n$, define
$$
\ltrmin(w) = \{ (i,w_i): w_j > w_i\text{ for every $j<i$}\}.
$$
Further, let the equivalence class of~$w$ be the set of permutations
$w'$ where $\ltrmin(w')=\ltrmin(w)$. Each equivalence class defined
this way contains exactly one permutation~$u$ avoiding $123$ and
one permutation~$v$ avoiding $132$; Simion-Schmidt's bijection
is the mapping $u\mapsto v$.
More explicitly, the entries of~$u$ and~$v$
are uniquely determined as follows. Let $w\in\Sym[n]$.
For $i=1,2,\ldots,n$, if $(i,a)\in \ltrmin(w)$, then let $u_i=v_i=a$. Otherwise:
\begin{itemize}
\item Let $u_i$ be equal to the largest letter in $[n]$ that has not been
used before.
\item Let $v_i$ be equal to the smallest letter in $[n]$ that has not been
used before and is greater than all the letters used thus far.
\end{itemize}
Picking the largest letter not used at every step guarantees that the
resulting permutation~$u$ avoids~$123$; indeed, a permutation avoids~$123$
if and only if the entries that are not left-to-right minima form a
decreasing sequence. Similarly, picking the smallest letter not used
ensures that~$v$ avoids~$132$. Finally, the additional constraint that
the letter picked at every step is greater than all the letters seen
before ensures that this construction gives the desired set of
left-to-right minima positions and values $\ltrmin(u)=\ltrmin(v)$.

As the reader may have guessed, we can modify the
Simion-Schmidt construction to obtain a bijection between $\Cay(123)[n]$
and $\Cay(132)[n]$ for $n\ge 0$.
Given $w\in\Cay[n]$, let
$$
\wltrmin(w) = \{ (i,w_i): w_j\ge w_i\text{ for every $j<i$}\}
$$
be the set of weak left-to-right minima positions and values of~$w$.
Further, let the \emph{filling} of $w$ be the multiset of its letters
that are not weak left-to-right minima
$$
\filling(w) = \{ w_i: (i,w_i)\notin \wltrmin(w)\}.
$$
Now, define the equivalence class of~$w$ as the set of Cayley
permutations $w'\in\Cay[n]$, where $\wltrmin(w')=\wltrmin(w)$ and
$\filling(w')=\filling(w)$. 
Once again, we shall see that each equivalence class defined this way contains
exactly one Cayley permutation $u\in\Cay(123)[n]$ and one Cayley permutation 
$v\in\Cay(132)[n]$, which are defined as follows.
Let $w\in\Cay[n]$.
For $i=1,2,\dots,n$, if~$(i,a)\in\wltrmin(w)$, then let $u_i=v_i=a$. Otherwise:
\begin{itemize}
\item Let $u_i$ be equal to the largest letter in $\filling(w)$ that has not been
used before.
\item Let $v_i$ be equal to the smallest letter in $\filling(w)$ that has not been
used before and is (strictly) greater than all the letters used thus far.
\end{itemize}
The additional requirement that $\filling(u)=\filling(v)$ is necessary as
otherwise, due to the possibility of having repeated entries, there could
be more than one $123$-avoiding Cayley permutation or more than one
$132$-avoiding Cayley permutation with the same positions and values of weak
left-to-right minima. To illustrate this construction, let
$$
w = \underline{7}\;\underline{7}\;9\;8\;\underline{5}\;9\;9\;\underline{5}\;6\;7\;\underline{4}\;\underline{1}\;2\;6\;3\;\underline{1}\;3\;3,
$$
where the underscored entries are the weak left-to-right minima of $w$. We have:
\begin{align*}
  \wltrmin(w) &= \{ (1,7),(2,7),(5,5),(8,5),(11,4),(12,1),(16,1) \};\\
  \filling(w) &= \{2,3,3,3,6,6,7,8,9,9,9\}.
\end{align*}
To obtain the only $123$-avoiding Cayley permutation~$u$ in the equivalence
class of~$w$, we keep the same positions and values of weak left-to-right minima
and arrange the letters of~$\filling(w)$ in weakly decreasing order:
$$
u=\underline{7}\;\underline{7}\;9\;9\;\underline{5}\;9\;8\;\underline{5}\;7\;6\;\underline{4}\;\underline{1}\;6\;3\;3\;\underline{1}\;3\;2.
$$
Finally, the corresponding $132$-avoiding Cayley permutation is
$$
v=\underline{7}\;\underline{7}\;8\;9\;\underline{5}\;6\;6\;\underline{5}\;7\;9\;\underline{4}\;\underline{1}\;2\;3\;3\;\underline{1}\;3\;9.
$$
To see that~$v$ avoids $132$, suppose for a contradiction that there are
three entries $v_i$, $v_j$, $v_k$ with $i<j<k$ and $v_i<v_k<v_j$.
Without loss of generality, we can assume that $(i,v_i)\in\wltrmin(v)$.
Note that $(j,v_j)$ and $(k,v_k)$ are not in $\ltrmin(v)$.
In particular, both~$v_j$ and~$v_k$ are in the filling of~$v$, which
contradicts $v_k<v_j$ since the procedure defined above
should pick the smaller entry~$v_k$ before~$v_j$.

Using the reverse and complement maps, together with the extension
of the Simion-Schmidt bijection to Cayley permutations shown above,
it follows that all permutations in $\Sym[3]=\{123,321,312,213,132,231\}$ are
Cayley-equivalent. An alternative way to obtain this result
consists of combining results by Chen, Dai and Zhou~\cite{chen}
and Kasraoui~\cite{KASRAOUI} to
obtain the following ordinary generating series for any $p\in\Sym[3]$:

\begin{equation}\label{eq: OGF Cay(231)}
\sum_{n\ge 0}\bigl|\Cay(p)[n]\bigr|x^n
=\frac{1}{2} + \frac{1}{1+\sqrt{1-8x+8x^2}}.
\end{equation}

We also have the explicit formula~\cite[Proposition~1.1]{KASRAOUI}:
\begin{equation}\label{eq: formula for Cay(123)}
|\Cay(p)[n]|=\sum_{k=1}^n\sum_{j=1}^k(-1)^{k-j}\binom{k}{j}|[j]^n(p)|.
\end{equation}
The number of $p$-avoiding words over an alphabet of~$k$ letters
was determined by Burstein~\cite[Theorem~3.2]{Bur1998}:
\begin{equation}\label{eq: Burstein}
|[k]^n(p)| =2^{n-2(k-2)}\sum_{m=0}^{k-2}\sum_{j=m}^{k-2}C_j\binom{2(k-2-j)}{k-2-j}\binom{n+2m}{n},
\end{equation}
where $C_j$ is the $j$th Catalan number. Combining formulas \eqref{eq:
  formula for Cay(123)} and \eqref{eq: Burstein} leads to a rather
unwieldy quadruple sum for $|\Cay(p)[n]|$. Luckily there is a more
compact way of expressing the result. Birmajer et
al.~\cite{Birmajer-etal} enumerated weak orderings subject to the
``stopping condition'' $x_{i_1}<x_{i_2}<x_{i_3}$. In our terminology
they showed that
\begin{equation}\label{eq:Birmajer-etal}
  \bigl| \Cay(123)[n] \bigr| = \sum_{j=0}^n(-1)^j2^{n-j-1}\binom{n-j}{j}C_{n-j}
\end{equation}
for $n\geq 1$, where $C_n$ is the $n$th Catalan number.

The results from Sections~\ref{sec: length two} and~\ref{sec: length three}
are summarized in Table~\ref{tab: patterns}.

\begin{table}
\centering
\tabulinesep=4mm \setlength{\tabcolsep}{2.5mm}
\begin{tabu}{@{}c|c|c|c|c@{}}
\noalign{\hrule height 2pt}
Pattern & Species & Series  & Enumeration $(n\geq 1)$ & OEIS \\
\noalign{\hrule height 1pt}
11 & $L$ & $\displaystyle \frac{1}{1-x}$ & $n!$ & \href{https://oeis.org/A000142}{A000142} \\
\noalign{\hrule height 1pt}
12 & \multirow{2}{*}{$\Even{E}\cdot E$} & \multirow{2}{*}{$\displaystyle \frac{e^{2x} +1}{2}$}
    & \multirow{2}{*}{$2^{n-1}$ }
    & \multirow{2}{*}{\href{https://oeis.org/A011782}{A011782}} \\
\cline{1-1}
21 & & & & \\
\noalign{\hrule height 1.2pt}
111 & $L(E_1 + E_2)$  & $\displaystyle \frac{2}{2-2x-x^2}$ &
$\displaystyle n!\cdot\frac{(1+\sqrt{3})^{n+1} -
(1-\sqrt{3})^{n+1}}{2^{n+1}\sqrt{3}}$ & \href{https://oeis.org/A080599}{A080599}  \\
\noalign{\hrule height 1pt}
212 & \multirow{6}{*}{$\Alt'$} &
\multirow{6}{*}{$\displaystyle\frac{x^2-2x+2}{2(x-1)^2}$} &
\multirow{6}{*}{$\displaystyle\frac{(n+1)!}{2}$}
& \multirow{6}{*}{\href{https://oeis.org/A001710}{A001710}}\\
\cline{1-1}
121 &  & & & \\
\cline{1-1}
112 &  & & & \\
\cline{1-1}
211 & &  & & \\
\cline{1-1}
221 & &  & & \\
\cline{1-1}
122 & &  & & \\
\noalign{\hrule height 1pt}
123 & \multirow{6}{*}{} &
\multirow{6}{*}{} &
\multirow{6}{*}{$\displaystyle\sum_{j=0}^n(-1)^j2^{n-j-1}\binom{n-j}{j}C_{n-j}$}
& \multirow{6}{*}{\href{https://oeis.org/A226316}{A226316}}\\
\cline{1-1}
132 &  & & & \\
\cline{1-1}
213 &  & & & \\
\cline{1-1}
231 & & & & \\
\cline{1-1}
321 & &  & & \\
\cline{1-1}
312 & &  & & \\
\noalign{\hrule height 2pt}
\end{tabu}
\caption{\label{tab: patterns}
Results for patterns of lengths two and three.}
\end{table}

\section{Primitive Cayley permutations}\label{sec: primitive}

A Cayley permutation $w=w_1\ldots w_n$ over $[n]$ is said to be
\emph{primitive} if it is nonempty $(n \geq 1)$ and
$w_i \neq w_{i+1}$ for $1 \leq i \leq n-1$.
That is, a primitive Cayley permutation has no ``flat steps''.
Under the designation \emph{multipermutations}, primitive Cayley permutations 
have figured in the work of Lam and
Pylyavskyy~\cite{LP2007}, as well as in the work of
Marberg~\cite{Marberg2021}.
We let $\Prim$ denote the $\LL$-species of primitive Cayley permutations.

\begin{lemma}\label{lemma:prim}
  $\Cay = 1 + \int (E\cdot \Prim')$.
\end{lemma}
\begin{proof}
  Let $F=1 + \int (E\cdot \Prim')$ and let $n$ be a nonnegative integer.
  We shall define a bijection $\alpha:F[n]\to\Cay[n]$. Since
  there is a unique $F$-structure on the empty set, which can be mapped
  to the empty Cayley permutation, we may assume that $n>0$.
  Now, an $F$-structure on $[n]$ is an
  $(E\cdot \Prim')$-structure on
  $[n]\setminus\{1\}=\{2,3,\dots,n\}$. Such a structure is a
  pair $(S, v)$, where $S$ is a subset of $\{2,3,\dots,n\}$ and $v$ is a
  primitive Cayley permutation on $1\oplus T = \{1\}\cup T$ with
  $T=\{2,3,\dots,n\}\setminus S$. Let $k$ be the number of letters of
  $v$ and write $v=v_1v_2\ldots v_k$. Then $\alpha(v)=w=w_1w_2\dots w_n$
  is the Cayley permutation obtained from filling $n$ slots as follows.
  Write down the letters $v_1$, $v_2$, \dots, $v_k$ of $v$ on the slots
  belonging to $T$. Note that $w_1 = v_1$. Moving from left to right,
  fill in the vacant slots by duplicating the letter to its left.  For
  instance, let $S=\{2,3,7\}$ and $v=325154$. We start by filling the
  slots $\{1,4,5,6,8,9\}$ with the letters of $v$, obtaining
  $3\,\_\,\_\,2\,5\,1\,\_\,5\,4$, then we fill in the vacant slots and
  arrive at $w=333251154$. It is easy to see how to reverse this process
  and thus $\alpha$ is a bijection.
\end{proof}

\begin{theorem}\label{thm:prim-cay}
  $\Prim' = \Cay^2$.
\end{theorem}
\begin{proof}
  Differentiating the equation given in Lemma~\ref{lemma:prim} yields
  $\Cay' = E\cdot \Prim'$. Next we solve for $\Prim'$ by multiplying by
  the inverse of $E$. It is the virtual species
  \[
    E^{-1} = (1+E_+)^{-1} = \sum_{k\geq 0} (-1)^kE_+^k.
  \]
  Thus, an $E^{-1}$-structure is a ballot such that if it has an even
  number of blocks it is considered positive, while if it has an odd
  number of blocks it is considered negative. Continuing with our
  derivation we get
  \begin{align*}
    \Prim'
    &= E^{-1}\cdot \Cay' \\
    &= E^{-1}\cdot \Bal' \\
    &= E^{-1}\cdot (L(E_+))' \\
    &= E^{-1}\cdot E_+'\cdot L^2(E_+) \\
    &= L^2(E_+) \\
    & = \Bal^2\\
    & = \Cay^2. \qedhere
  \end{align*}
\end{proof}

The OEIS entry for the coefficients of $\Cay^2(x)$ is A005649.

A consequence of Lemma~\ref{lemma:prim} is 
$\Cay'=E\cdot \Prim'$. Inspired by this identity, we more generally say
that if two species $F$ and $G$ are related by $F(0)=1$, $G(0)=0$, and
\[
  F' = E\cdot G',
\]
then $G$ is the species of \emph{primitive $F$-structures}. For $\LL$-species, this is equivalent to
\[
F=1+\int (E\cdot G').
\]
Examples include:
\begin{itemize}
\item Primitive $\Cay$ is $\int \Cay^2$ (above).
\item Primitive $\Par$ is $\int \Par$ (similar to above).
\item Primitive (modified) ascent sequences are nonempty (modified) ascent
sequences with no flat steps~\cite{modasc,primitiveascseq}.
\item Primitive $\Sym$ is $\int (\Der + \Der') = \int\Der + \Der_+$ as follows from
differentiating $\Sym=E\cdot \Der$, in which $\Der$ is the species of derangements. 
Alternatively, Claesson~\cite{Cl2022} proved that primitive permutations are
the integral of those avoiding the Hertzsprung pattern $12$.
\end{itemize}

The equation $F'=E\cdot G'$ yields the formula
\begin{align}
  |F[n]| &= \sum_{j=1}^{n}\binom{n-1}{j-1} |G[j]| &&\text{ for } n\geq 1, \label{eq: F from G}\\
  \intertext{where $|F[0]|=1$. Similarly, the equation $G'=E^{-1}\cdot F'$ yields the formula}
  |G[n]| &= \sum_{j=1}^{n}(-1)^{n-j}\binom{n-1}{j-1} |F[j]| &&\text{ for } n\geq 0. \label{eq: G from F}
\end{align}

Let $\Prim(p)$ be the species of primitive Cayley permutations avoiding
the pattern $p$. The proof of the following proposition is almost
identical to that of Lemma~\ref{lemma:prim} and we omit the
details. Note, however, that it is crucial for $p$ to be primitive.
Otherwise, an occurrence of~$p$ could be created in the process of
duplicating the letters of a $\Prim(p)'$-structure.

\begin{proposition}\label{prop: Prim(p)}
  If $p$ is a primitive Cayley permutation, then
  \begin{align*}
    \Cay(p) &= 1 + \int (E\cdot \Prim(p)') \\
    \shortintertext{and}
    \Prim(p) &= \int (E^{-1}\cdot \Cay(p)').
  \end{align*}
\end{proposition}

For instance, $\Cay(21)'=E^2$ by Proposition~\ref{prop: Cay(21) species}
and hence
\[
  \Prim(21) = \int (E^{-1}\cdot \Cay(21)') = \int E = E_+.
\]
This should come as no surprise since, for each $n \geq 1$, the only
primitive $21$-avoiding Cayley permutation on~$[n]$ is the identity
permutation $12\ldots n$.

For another instance, $212$ is a primitive pattern and
$\Cay(212) = \Alt'$ by Theorem~\ref{thm:species 212}. Thus
$\Prim(212) = \int (E^{-1}L^3)$ and the corresponding generating series is
\begin{align*}
  \Prim(212)(x)
  &= \int_0^x \frac{e^{-t}}{(1-t)^3}\,dt \\
  &= x + 2\cdot\frac{x^2}{2!} + 7\cdot\frac{x^3}{3!} + 32\cdot\frac{x^4}{4!}
    + 181\cdot\frac{x^5}{5!} + 1214\cdot \frac{x^6}{6!} + \cdots
\end{align*}
whose coefficients form sequence A000153 in the OEIS.

Since any permutation is a primitive Cayley permutation,
Proposition~\ref{prop: Prim(p)} applies whenever the pattern $p$ is a
permutation. Consider $p=231$, or any other permutation in $\Sym[3]$; they are all
Cayley-equivalent by the work in Section~\ref{sec: length three}.
To calculate
\begin{equation}\label{eq: Prim(231) from Cay(231)}
  \Prim(231)(x) = \int_0^xe^{-t}\Cay(231)(t)\,dt
\end{equation}
we would need to know $\Cay(231)(x)$, and it may be possible
to find an expression for $\Cay(231)(x)$ by solving the
equation in Proposition~\ref{prop: equation for cay231}, but we have
failed to do so.
It is, however, easy to calculate the numbers $\bigl|\Cay(231)[n]\bigr|$;
e.g., using formula~\eqref{eq:Birmajer-etal}. The result is sequence A226316:
\[
  1, 1, 3, 12, 56, 284, 1516, 8384, 47600, 275808, 1624352, 9694912, \ldots
\]
Using equation~\eqref{eq: G from F} we may calculate the corresponding
numbers $\bigl|\Prim(231)[n]\bigr|$ for primitive Cayley
permutations. They turn out to be
\[
  0,1,2,7,28,121,550,2591,12536,61921,310954,1582791, \ldots
\]
and match sequence A010683, whose ordinary generating series is the
square of that for the Schröder–Hipparchus numbers (A001003), shifted
by one. We have arrived at an educated guess:
\begin{equation}\label{guess}
  \sum_{n\geq 0}\bigl|\Prim(231)[n]\bigr| x^n \,=\,
  x\left( \frac{2}{1 + x + \sqrt{1 - 6x + x^2}} \right)^2.
\end{equation}
Can we prove this? Yes, but first we need to translate the
relationship~\eqref{eq: Prim(231) from Cay(231)} for the exponential
generating series to a relationship between ordinary generating
series. For this discussion, let $a_0, a_1, a_2, \dots$ be a sequence of
numbers and let $A(x)$ and $\hat{A}(x)$ be the corresponding exponential
and ordinary generating series. Furthermore, let $b_0, b_1, b_2, \dots$
be another sequence of numbers and define $B(x)$ and $\hat{B}(x)$
in the same manner. The integral ``transform'', $A(x)\mapsto \int_0^x A(t)\, dt$,
corresponds to a shift of the coefficients, $a_n\mapsto a_{n-1}$, which
in turn corresponds to $\hat{A}(x)\mapsto x\hat{A}(x)$. The derivative,
$A(x)\mapsto A'(x)$, gives a shift in the other direction,
$a_n\mapsto a_{n+1}$, which corresponds to
$\hat{A}(x)\mapsto (\hat{A}(x)-1)/x$. Finally, the so-called inverse
binomial transform, $a_n = \sum_{j=0}^n(-1)^{n-j}\binom{n}{j}b_k$, or
$B(x)=e^{-x}A(x)$, corresponds to
$\hat{A}(x)=\hat{B}\bigl(x/(1+x)\bigr)/(1+x)$.  Putting this all
together we can translate equation~\eqref{eq: Prim(231) from Cay(231)}
and the result is
\begin{align*}
  \sum_{n\geq 0}\bigl|\Prim(231)[n]\bigr| x^n
  &\,=\, \frac{x}{1+x}\left[\frac{1}{t}\left(\sum_{n\geq 0}\bigl|\Prim(231)[n]\bigr| t^n - 1\right) \right]_{t=\frac{x}{1+x}} \\
  &\,=\, \frac{1+x}{1+x+\sqrt{1-6x+x^2}} - \frac{1}{2},
\end{align*}
where we have used \eqref{eq: OGF Cay(231)} to
calculate the explicit expression. It is easy to verify that this answer
is consistent with our guess~\eqref{guess}.

Despite its simplicity, we could not find a bijective proof of the
species equality $\Prim'=\Cay^2$ of Theorem~\ref{thm:prim-cay}.
The interplay between $\Prim'$ and $\Cay$ is not limited to this
equality. Indeed, we also conjecture the following two equalities.

\begin{conjecture}
For all $n\geq 1$, we have
$$
\frac{1}{2}|\Prim'[n]|
\,=\, \frac{1}{n} \sum_{w\in\Cay[n]}\sum_{i=1}^n w_i
\,= \sum_{w\in\Cay[n]} \sum_{i\in\textsc{fix}(w)}i,
$$
where $\textsc{fix}(w)=\{ i: w_i=i\}$ is the set of fixed points of~$w$.
\end{conjecture}

\section{Wilf equivalences}\label{section: wilfeq}

Recall from Section~\ref{sec:pattsonCayley} that two Cayley
permutations~$p$ and~$q$ are Cayley-equivalent if
$|\Cay(p)[n]|=|\Cay(q)[n]|$ for every $n\ge 0$. The following proposition 
summarizes the Cayley-equivalence results from Section~\ref{sec: length three}.

\begin{proposition}
Cayley-equivalence partitions patterns of length three into the following three classes:
\begin{itemize}
\item $\{111\}$;
\item $\{112,121,122,211,212,221\}$;
\item $\{123,132,213,231,312,321\}$.
\end{itemize}
\end{proposition}

Jel\'inek and Mansour~\cite{k-ary} considered two alternative notions
of equivalence on $k$-ary words: strong (word) equivalence and word
equivalence. We define them below, together with their natural
counterparts on Cayley permutations. We also introduce an analogue
to Cayley-equivalence on $k$-ary words. We expand a bit on the
topic of Wilf equivalences on words and Cayley permutations by relating
some of the equivalences defined this way. Finally, we state some
open problems for future investigation.

Let the \emph{content} of a word be the multiset of its
letters. Further, denote by $\Cay^k(p)[n]$ the set of Cayley
permutations over~$[n]$ that avoid~$p$ and whose maximum value is
equal to~$k$. Two Cayley permutations~$p$ and~$q$ are
\begin{itemize}
\item \emph{strong-word-equivalent} ($p\sim_{sw} q$) if for
every $k,n$ there is a bijection between $[k]^n(p)$ and $[k]^n(q)$
that preserves the content;
\item \emph{word-equivalent} ($p\sim_{w} q$) if $|[k]^n(p)|=|[k]^n(q)|$
for every $k,n$;
\item \emph{endo-equivalent} ($p\sim_{e} q$) if
$|[n]^n(p)|=|[n]^n(q)|$ for every $n$,
\end{itemize}
where ``endo'' is short for endofunction. Similarly, $p$ and $q$ are
\begin{itemize}
\item \emph{strong-Cayley-equivalent} ($p\sim_{sc} q$) if for
every $k,n$ there is a bijection between $\Cay^k(p)[n]$ and $\Cay^k(q)[n]$
that preserves the content;
\item \emph{Cayley-max-equivalent} ($p\sim_{cm} q$) if
$|\Cay^k(p)[n]|=|\Cay^k(q)[n]|$ for every $k,n$;
\item \emph{Cayley-equivalent} ($p\sim_{c} q$) if
$|\Cay(p)[n]|=|\Cay(p)[n]|$ for every $n$.
\end{itemize}

Clearly, strong-word-equivalence implies word-equivalence, which in
turns implies endo-equivalence. Similarly, strong-Cayley-equivalence
implies Cayley-max-equival\-ence, which in turn implies Cayley-equivalence.
Furthermore, Kasraoui~\cite[Proposition~1.1]{KASRAOUI} showed
that if~$p$ and~$q$ are permutations and $p\sim_w q$, then~$p\sim_c q$.
We shall prove that strong-word-equivalence coincides with
strong-Cayley-equivalence, and that word-equivalence coincides with
Cayley-max-equivalence.
To do that, let us consider the following \emph{standardization} map.
Let $w\in [n]^n$ and let~$A$ be a subset of $[n]$ with $|A|=|\Img(w)|$.
Then $\std_A(w)$ is the word obtained by replacing each copy of
the $i$th smallest letter of~$w$ with the $i$th smallest element of~$A$.
As an example, let $w=337217813\in [9]^9$ and let $A=\{2,5,6,7,9\}$.
Then $\std_A(w)=667527926$.
By choosing $A=[k]$ with $k=|\Img(w)|$, we obtain a Cayley permutation
$\std_A(w)$ whose maximum value is equal to~$k$. Moreover,
the standardization map preserves pattern avoidance and containment;
more explicitly, for any pattern~$p$ we have that~$w$ contains~$p$
if and only if $\std_A(w)$ contains~$p$. 

\begin{lemma}\label{lemma: eqforwilfequiv}
For each Cayley permutation $p$ and for each $k,n\ge 0$, we have
\begin{align}
|[k]^n(p)| &= \sum_{i=0}^k \binom{k}{i}|\Cay^i(p)[n]|\label{eq1wilfeq}\\
\shortintertext{and}
|\Cay^k(p)[n]| &= \sum_{i=0}^k(-1)^{k-i}\binom{k}{i}|[i]^n(p)|.\label{eq2wilfeq}
\end{align}
\end{lemma}
\begin{proof}
Let $n,k\ge 0$. For any given $i$, the map
$w\mapsto\bigl(\Img(w),\std_{[i]}(w)\bigr)$ is a bijection
\[
  \bigl\{ w\in [k]^n(p): |\Img(w)|=i \bigr\}
  \;\longrightarrow\;
  \binom{[k]}{i}\times \bigl\{w\in\Cay^i(p)[n]:\max(w)=i\bigr\}.
\]
Its inverse is given by $(A,v)\mapsto\std_A(v)$.
Thus,
\[
\bigl|[k]^n(p)\bigr|
\;=\; \sum_{i=0}^k \bigl|\bigl\{ w\in [k]^n(p): |\Img(w)|=i\bigr\}\bigr|
\;=\; \sum_{i=0}^k \binom{k}{i}\bigl|\Cay^i(p)[n]\bigr|,
\]
proving equation~\eqref{eq1wilfeq}. This relation between
the numbers $|\Cay^i(p)[n]|$ and $|[k]^n(p)|$ can easily be
inverted to obtain~\eqref{eq2wilfeq}. Indeed, letting
$A(x)=\sum_{k\ge 0}|\Cay^k(p)[n]|x^k/k!$ and
$B(x)=\sum_{k\ge 0}|[k]^n(p)|x^k/k!$, equation~\eqref{eq1wilfeq}
amounts to $B(x) = e^xA(x)$, or, equivalently, $A(x)=e^{-x}B(x)$, which gives
equation~\eqref{eq2wilfeq}.
\end{proof}

\begin{proposition}\label{prop: caywilfeq}
Let~$p$ and~$q$ be Cayley permutations. Then
$$
p\sim_w q\;\iff\;p\sim_{cm}q.
$$
\end{proposition}
\begin{proof}
If $p\sim_w q$, then $|[i]^n(p)|=|[i]^n(q)|$ for every~$i$
and thus $p\sim_{cm} q$ by equation~\eqref{eq2wilfeq}.
Similarly, if $p\sim_{cm} q$, then $|\Cay^k(p)[n]|=|\Cay^k(q)[n]|$
for every~$k$, and $p\sim_{w} q$ by equation~\eqref{eq1wilfeq}.
\end{proof}

\begin{proposition}\label{prop: strongwilfeq}
Let~$p$ and~$q$ be Cayley permutations. Then
$$
p\sim_{sw} q\;\iff\;p\sim_{sc}q.
$$
\end{proposition}
\begin{proof}
Suppose first that $p\sim_{sw} q$. That is, for each $n,k$ there is
a content-preserving bijection
$$
f_{k,n}:\; [k]^n(p)\longmapsto [k]^n(q).
$$
Then, we obtain a content-preserving bijection $g_{k,n}$ between
$\Cay^k(p)[n]$ and $\Cay^k(q)[n]$ by simply letting $g_{k,n}$ be
the restriction of $f_{k,n}$ to the subset $\Cay^k(p)[n]$ of $[k]^n(p)$.
This proves $p\sim_{sc} q$.
To prove the converse implication, suppose that $p\sim_{sc} q$. Equivalently,
for each $n,k$ we have a content-preserving bijection
$$
g_{k,n}:\; \Cay^k(p)[n]\longmapsto \Cay^k(q)[n].
$$
Then we define a content-preserving bijection $f_{k,n}$ from $[k]^n(p)$
to $[k]^n(q)$ by letting
$$
f_{k,n}(w)=
\left(\std_{\Img(w)}\circ g_{j,n}\circ\std_{[j]}\right)(w),
$$
where $j=|\Img(w)|$.
In other words, $f_{k,n}(w)$ is obtained by first standardizing~$w$
under $\std_{[j]}$, then applying the suitable content-preserving
bijection $g_{j,n}$, and finally applying the inverse standardization
(i.e., $\std_{\Img(w)}$) to obtain a word that has the same content as~$w$.
Or, even less formally, by applying $g_{j,n}$ to~$w$ pretending that
the numbers that appear in $w$ are $1,2,\dots,j$.
It is easy to see that $f_{k,n}$ preserves the content, as well as the
avoidance of~$p$. The proof that $f_{k,n}$ is a bijection is left to the
reader. 
\end{proof}

We end this section with a list of open problems.
Let $p$ and $q$ be two Cayley permutations of the same length. Can we prove
or refute any of the following conjectures?

\begin{conjecture}\label{conj: caystrongeq}
If $p\sim_{cm} q$, then $p\sim_{sc} q$.
\end{conjecture}

Jel\'inek and Mansour~\cite{k-ary} showed that strong-word-equivalence
and word-equivalence coincide on patterns of length at most six.
By Propositions~\ref{prop: caywilfeq} and~\ref{prop: strongwilfeq},
Conjecture~\ref{conj: caystrongeq} holds up to length six.

\begin{conjecture}\label{conj: cayleyeq}
If $p\sim_{c} q$, then $p\sim_{cm} q$.
\end{conjecture}

The smallest candidate for a counterexample to Conjecture~\ref{conj: cayleyeq}
is the pair $p=13442$ and $q=14233$. Here,
$p\not\sim_{cm} q$ since e.g.,
$|\Cay^5(p)[9]|=742943\neq 742944=|\Cay^5(q)[9]|$. If the conjecture is true, then
$|\Cay(p)[n]|$ and $|\Cay(q)[n]|$ will differ for some $n$. We have, however, checked that
$|\Cay(p)[n]|=|\Cay(q)[n]|$ for $n\le 9$.

It is easy to see that if $\max(p)\neq\max(q)$, then $p\not\sim_{cm}q$.
Indeed, suppose that $\max(p)=m<\max(q)$, for $p,q\in\Cay[n]$. Then there
is only one Cayley permutation in $\Cay^m[n]$ that contains~$p$, namely~$p$
itself; on the other hand, no Cayley permutation in $\Cay^m[n]$
contains~$q$ since we assumed $\max(q)>m$. Thus,
$$
|\Cay^m(p)[n]| = |\Cay^m[n]|-1 < |\Cay^m[n]| = |\Cay^m(q)[n]|,
$$
showing that $p\not\sim_{cm}q$. Since $p\sim_{sc}q$ implies
$p\sim_{cm}q$, the same property holds for strong-Cayley-equivalence.
What about Cayley-equivalence?

\begin{conjecture}
  If $p\sim_c q$, then $\max(p)=\max(q)$.
\end{conjecture}

Finally, we conjecture that $\max(p)\le\max(q)$ turns into the opposite
inequality for the number of Cayley permutations avoiding the two patterns.

\begin{conjecture}
If $\max(p)\leq \max(q)$, then $|\Cay_n(p)| \geq |\Cay_n(q)|$.
\end{conjecture}


\bibliographystyle{plain}
\bibliography{references.bib}

\begin{thebibliography}{10}

\bibitem{AUP2013}
Connor Ahlbach, Jeremy Usatine, and Nicholas Pippenger.
\newblock Barred preferential arrangements.
\newblock {\em The Electronic Journal of Combinatorics}, 20(2), 2013.

\bibitem{Bailey1998}
Ralph~W. Bailey.
\newblock The number of weak orderings of a finite set.
\newblock {\em Social Choice and Welfare}, 15(4):559--562, 1998.

\bibitem{Baril2003}
Jean-Luc Baril.
\newblock Gray code for cayley permutations.
\newblock {\em Computer Science}, 11(2):32, 2003.

\bibitem{holybook}
Fran{\c{c}}ois Bergeron, Gilbert Labelle, and Pierre Leroux.
\newblock {\em Combinatorial species and tree-like structures}.
\newblock Number~67. Cambridge University Press, 1998.

\bibitem{BS1995}
M.~Bernstein and N.~J.~A. Sloane.
\newblock Some canonical sequences of integers.
\newblock {\em Linear Algebra and its Applications}, 226/228:57--72, 1995.

\bibitem{permpatternsbasicdefinitions}
David Bevan.
\newblock Permutation patterns: basic definitions and notation.
\newblock {\em \href{https://arxiv.org/abs/1506.06673}{\tt arXiv:1506.06673}},
  2015.

\bibitem{BZG2016}
Yonah Biers-Ariel, Yiguang Zhang, and Anant Godbole.
\newblock Some results on superpatterns for preferential arrangements.
\newblock {\em Advances in Applied Mathematics}, 81:202--211, 2016.

\bibitem{BilleyRyan}
Sara~C. Billey and Stark Ryan.
\newblock Brewing fubini-bruhat orders.
\newblock {\em S{\'e}minaire lotharingien de combinatoire}, in press.
\newblock Proceedings of the 36th Conference on Formal Power Series and
  Algebraic Combinatorics (Bochum).

\bibitem{Birmajer-etal}
Daniel Birmajer, Juan~B. Gil, David~S. Kenepp, and Michael~D. Weiner.
\newblock {Restricted generating trees for weak orderings}.
\newblock {\em {Discrete Mathematics \& Theoretical Computer Science}}, {vol.
  24, no. 1}, March 2022.

\bibitem{Bur1998}
Alexander Burstein.
\newblock {\em {Enumeration of words with forbidden patterns}}.
\newblock {PhD Thesis}, University of Pennsylvania, 1998.

\bibitem{Bona2022}
Miklos Bóna.
\newblock {\em {Combinatorics of Permutations}}.
\newblock Chapman and Hall/CRC, 3rd edition, 2022.

\bibitem{Cayley1859}
Arthur Cayley.
\newblock On the analytic forms called trees, second part.
\newblock {\em Philosophical Magazine}, 18:374--378, 1859.

\bibitem{sortcay}
Giulio Cerbai.
\newblock Sorting {C}ayley permutations with pattern-avoiding machines.
\newblock {\em The Australasian Journal of Combinatorics}, 80:322--341, 2021.

\bibitem{modasc}
Giulio Cerbai.
\newblock Modified ascent sequences and {B}ell numbers.
\newblock {\em \href{https://arxiv.org/abs/2305.10820}{\tt arXiv:2305.10820}},
  2023.

\bibitem{modasc2}
Giulio Cerbai.
\newblock Pattern-avoiding modified ascent sequences.
\newblock {\em \href{https://arxiv.org/abs/2401.10027}{\tt arXiv:2401.10027}},
  2024.

\bibitem{CC:fish}
Giulio Cerbai and Anders Claesson.
\newblock Fishburn trees.
\newblock {\em Advances in Applied Mathematics}, 151, 2023.

\bibitem{CC:tpb}
Giulio Cerbai and Anders Claesson.
\newblock Transport of patterns by {B}urge transpose.
\newblock {\em European Journal of Combinatorics}, 108, 2023.

\bibitem{CC:caypol}
Giulio Cerbai and Anders Claesson.
\newblock Caylerian polynomials.
\newblock {\em Discrete Mathematics}, 347(12), 2024.

\bibitem{modasc3}
Giulio Cerbai, Anders Claesson, and Bruce~E. Sagan.
\newblock Modified difference ascent sequences and {F}ishburn structures.
\newblock {\em \href{https://arxiv.org/abs/2406.12610}{\tt arXiv:2406.12610}},
  2024.

\bibitem{chen}
William~Y.C. Chen, Alvin~Y.L. Dai, and Robin~D.P. Zhou.
\newblock Ordered partitions avoiding a permutation pattern of length 3.
\newblock {\em European Journal of Combinatorics}, 36:416--424, 2014.

\bibitem{12foldway}
Anders Claesson.
\newblock A species approach to {R}ota’s twelvefold way.
\newblock {\em Expositiones Mathematicae}, 2019.

\bibitem{Cl2022}
Anders Claesson.
\newblock From {H}ertzsprung's problem to pattern-rewriting systems.
\newblock {\em Algebraic Combinatorics}, 5(6):1257--1277, 2022.

\bibitem{ClKi2008}
Anders Claesson and Sergey Kitaev.
\newblock Classification of bijections between 321- and 132-avoiding
  permutations.
\newblock {\em S\'{e}minaire Lotharingien de Combinatoire}, 60, 2008.

\bibitem{Conway2022}
Andrew~R. Conway, Miles Conway, Andrew~Elvey Price, and Anthony~J. Guttmann.
\newblock Pattern-avoiding ascent sequences of length 3.
\newblock {\em The Electronic Journal of Combinatorics}, 29(4), 2022.

\bibitem{primitiveascseq}
Mark Dukes, Sergey Kitaev, Jeffrey~B. Remmel, and Einar Steingr\'imsson.
\newblock Enumerating (2+2)-free posets by indistinguishable elements.
\newblock {\em Journal of Combinatorics}, 2(1):139--163, 2011.

\bibitem{Duncan2011}
Paul Duncan and Einar Steingr\'imsson.
\newblock Pattern avoidance in ascent sequences.
\newblock {\em The Electronic Journal of Combinatorics}, 18(1), 2011.

\bibitem{Egge2022}
Eric~S. Egge.
\newblock Pattern-avoiding {F}ishburn permutations and ascent sequences.
\newblock {\em \href{https://arxiv.org/abs/2208.01484}{\tt arXiv:2208.01484}},
  2022.

\bibitem{FNJ2014}
Lo\"ic Foissy, Jean-Christophe Novelli, and Jean-Yves Thibon.
\newblock Polynomial realizations of some combinatorial {H}opf algebras.
\newblock {\em Journal of Noncommutative Geometry}, 8(1):141--162, 2014.

\bibitem{Fraenkel1985}
Aviezri~S. Fraenkel.
\newblock Systems of numeration.
\newblock {\em American Mathematical Monthly}, 92(2):105--114, 1985.

\bibitem{FM1983}
Aviezri~S. Fraenkel and Moshe Mor.
\newblock Combinatorial compression and partitioning of large dictionaries.
\newblock {\em The Computer Journal}, 26(4):336--343, 1983.

\bibitem{GS78}
Ira Gessel and Richard~P. Stanley.
\newblock Stirling polynomials.
\newblock {\em Journal of Combinatorial Theory, Series A}, 24(1):25--33, 1978.

\bibitem{Gobel1997}
Manfred G\"obel.
\newblock On the number of special permutation-invariant orbits and terms.
\newblock {\em Applicable Algebra in Engineering, Communication and Computing},
  8(6):505--509, 1997.

\bibitem{Godbole2014}
Anant Godbole, Adam Goyt, Jennifer Herdan, and Lara Pudwell.
\newblock Pattern avoidance in ordered set partitions.
\newblock {\em Annals of Combinatorics}, 18, 2014.

\bibitem{Golab2024}
Hannah Golab.
\newblock {\em {Pattern Avoidance in Cayley Permutations}}.
\newblock {MS Thesis}, Northern Arizona University, 2024.

\bibitem{Gross1962}
Oliver~A. Gross.
\newblock Preferential arrangements.
\newblock {\em American Mathematical Monthly}, 69:4--8, 1962.

\bibitem{Hazewinkel2008}
Michiel Hazewinkel.
\newblock Hopf algebras of endomorphisms of {H}opf algebras.
\newblock {\em The Arabian Journal for Science and Engineering, Section C.
  Theme Issues}, 33(2):239--272, 2008.

\bibitem{Heubach2006}
Silvia Heubach and Toufik Mansour.
\newblock Avoiding patterns of length three in compositions and multiset
  permutations.
\newblock {\em Advances in Applied Mathematics}, 36(2), 2006.

\bibitem{Hoffman2001}
Michael~E. Hoffman.
\newblock An analogue of covering space theory for ranked posets.
\newblock {\em The Electronic Journal of Combinatorics}, 8(1), 2001.

\bibitem{Hoffman2016}
Michael~E. Hoffman.
\newblock Updown categories: generating functions and universal covers.
\newblock {\em Discrete Mathematics}, 339(2):906--922, 2016.

\bibitem{HH2013}
Victoria Horan and Glenn Hurlbert.
\newblock Universal cycles for weak orders.
\newblock {\em SIAM Journal on Discrete Mathematics}, 27(3):1360--1371, 2013.

\bibitem{Jacques-etal}
Marsden Jacques, Dennis Wong, and Kyounga Woo.
\newblock Generating {G}ray codes for weak orders in constant amortized time.
\newblock {\em Discrete Mathematics}, 343(10):111992, 10, 2020.

\bibitem{Jelinek2013}
V{\'\i}t Jel{\'\i}nek, Toufik Mansour, and Mark Shattuck.
\newblock On multiple pattern avoiding set partitions.
\newblock {\em Advances in Applied Mathematics}, 50(2), 2013.

\bibitem{k-ary}
Vít Jel\'inek and Toufik Mansour.
\newblock Wilf-equivalence on $k$-ary words, compositions, and parking
  functions.
\newblock {\em The Electronic Journal of Combinatorics}, 16(1), 2009.

\bibitem{KASRAOUI}
Anisse Kasraoui.
\newblock Pattern avoidance in ordered set partitions and words.
\newblock {\em Advances in Applied Mathematics}, 61:85--101, 2014.

\bibitem{Kitaev2011}
Sergey Kitaev.
\newblock {\em {Patterns in permutations and words}}.
\newblock Springer-Verlag, 1st edition, 2011.

\bibitem{Knuth1969}
Donald~E. Knuth.
\newblock {\em {The Art of Computer Programming}}, volume~1.
\newblock Addison-Wesley, 1969.

\bibitem{KR2022}
Dani\"el Kroes and Brendon Rhoades.
\newblock Packed words and quotient rings.
\newblock {\em Discrete Mathematics}, 345(9), 2022.

\bibitem{LP2007}
Thomas Lam and Pavlo Pylyavskyy.
\newblock Combinatorial {H}opf algebras and {$K$}-homology of {G}rassmannians.
\newblock {\em International Mathematics Research Notices. IMRN}, (24), 2007.

\bibitem{MacMahon1890}
Percy~A. MacMahon.
\newblock Yoke-{C}hains and {M}ultipartite {C}ompositions in connexion with the
  {A}nalytical forms called `{T}rees'.
\newblock {\em Proceedings of the London Mathematical Society}, 22:330--346,
  1890/91.

\bibitem{MacMahon1915}
Percy~A. MacMahon.
\newblock {\em {Combinatory Analysis}}, volume~1.
\newblock Cambridge University Press, 1915.

\bibitem{Marberg2020}
Eric Marberg.
\newblock Bialgebras for {S}tanley symmetric functions.
\newblock {\em Discrete Mathematics}, 343(4), 2020.

\bibitem{Marberg2021}
Eric Marberg.
\newblock Linear compactness and combinatorial bialgebras.
\newblock {\em The Electronic Journal of Combinatorics}, 28(3), 2021.

\bibitem{Mendelson1982}
Elliott Mendelson.
\newblock Races with ties.
\newblock {\em Mathematics Magazine}, 55(3):170--175, 1982.

\bibitem{MorFraenkel1984}
Moshe Mor and Aviezri~S. Fraenkel.
\newblock Cayley permutations.
\newblock {\em Discrete Mathematics}, 48(1):101--112, 1984.

\bibitem{Mutze2023}
Torsten M\"utze.
\newblock Combinatorial {G}ray codes---an updated survey.
\newblock {\em The Electronic Journal of Combinatorics}, DS26, 2023.

\bibitem{NBCC2020}
Sithembele Nkonkobe, Be\'ata B\'enyi, Roberto~B. Corcino, and Cristina~B.
  Corcino.
\newblock A combinatorial analysis of higher order generalised geometric
  polynomials: a generalisation of barred preferential arrangements.
\newblock {\em Discrete Mathematics}, 343(3):111729, 12, 2020.

\bibitem{oeis}
{OEIS Foundation Inc.}
\newblock {The {O}n-{L}ine {E}ncyclopedia of {I}nteger {S}equences}.
\newblock {\em Published electronically at {\tt https://oeis.org/}}, 2024.

\bibitem{PP2016}
Rebecca Patrias and Pavlo Pylyavskyy.
\newblock Combinatorics of {$K$}-theory via a {$K$}-theoretic
  {P}oirier-{R}eutenauer bialgebra.
\newblock {\em Discrete Mathematics}, 339(3):1095--1115, 2016.

\bibitem{Pippenger2010}
Nicholas Pippenger.
\newblock The hypercube of resistors, asymptotic expansions, and preferential
  arrangements.
\newblock {\em Mathematics Magazine}, 83(5):331--346, 2010.

\bibitem{Rhoades2022}
Brendon Rhoades.
\newblock Generalizations of the flag variety tied to the {M}acdonald-theoretic
  delta operators.
\newblock {\em \href{https://arxiv.org/abs/2204.03386}{\tt arXiv:2204.03386}},
  2022.

\bibitem{Sagan2006}
Bruce~E. Sagan.
\newblock Pattern avoidance in set partitions.
\newblock {\em Ars Combinatoria}, 94, 2006.

\bibitem{Savage2006}
Carla~D. Savage and Herbert~S. Wilf.
\newblock Pattern avoidance in compositions and multiset permutations.
\newblock {\em {Advances in Applied Mathematics}}, 36(2):194--201, 2006.

\bibitem{SF2022}
Alexander Schnurr and Svenja Fischer.
\newblock Generalized ordinal patterns allowing for ties and their applications
  in hydrology.
\newblock {\em Computational Statistics \& Data Analysis}, 171, 2022.

\bibitem{simionschmidt}
Rodica Simion and Frank~W. Schmidt.
\newblock Restricted permutations.
\newblock {\em European Journal of Combinatorics}, 6:383--406, 1985.

\bibitem{Wilson2018}
Andrew~T. Wilson.
\newblock Torus link homology and the nabla operator.
\newblock {\em Journal of Combinatorial Theory, Series A}, 154:129--144, 2018.

\end{thebibliography}

\end{document}